\documentclass[twoside,11pt]{amsart}
\usepackage{latexsym}
\usepackage{amssymb}
\usepackage{epsfig}
\usepackage{multirow}

\newtheorem{thm}{Theorem}[section]

\newtheorem{lem}[thm]{Lemma}

\theoremstyle{definition}

\theoremstyle{remark}

\newcommand{\R}{\mathbb{R}}

\makeatletter
\def\serieslogo@{}
\makeatother \makeatletter
\def\@setcopyright{}
\makeatother

\def\Real{{\mathbb R}}

\newcommand{\eq}[1]{{(\ref{#1})}}

\begin{document}
\title[]{ Gaussian beam methods for the Helmholtz equation}
\author{Hailiang Liu$^1$ \and James Ralston$^2$ \and Olof Runborg$^3$ \and Nicolay M. Tanushev$^4$}
\thanks{$^1$ Department of Mathematics, Iowa State University, Ames, IA 50010, USA}
\thanks{$^2$ Department of Mathematics, University of California at Los Angeles, Los Angeles, CA 90095, USA}
\thanks{$^3$ Department of Mathematics and Swedish e-Science Research Center (SeRC), KTH, 10044 Stockholm, Sweden}
\thanks{$^4$ Z-Terra Inc., 
17171 Park Row, Suite 247,
Houston TX 77084, USA}
\date{\today}

\subjclass[2000]{35B45, 35J05, 35Q60, 78A40}
\keywords{Helmholtz equation, high frequency wave
propagation, localized source, radiation
condition}

\begin{abstract}
In this work we construct Gaussian beam
approximations to solutions of the high frequency
Helmholtz equation with a localized source.
{
Under the assumption of non-trapping rays we show 
error estimates between the exact outgoing solution
and Gaussian beams in terms of the wave number $k$, 
both for single beams and superposition
of beams. The main result is that the relative local $L^2$ error
in the beam approximations decay as {$k^{-N/2}$
independent of dimension and presence of caustics,
for $N$-th order beams.}
}
\end{abstract}

\maketitle

\section{Introduction}

In this article we are interested in the accuracy
of Gaussian beam approximations to solutions of
the high frequency Helmholtz equation with a
source term,
\begin{align} \label{ho}
L_n u =_{def}
\Delta u+ (i\alpha k + k^2)n^2 u= f, \qquad x \in \mathbb{R}^d.
\end{align}
Here $k>0$ is the wave number, assumed to be
large, $n(x)$ is the index of refraction and
$f(x;k)$ is a source function which in general
also depends on $k$. We assume that both
$f(x;k)$ and $n(x)-1$ vanish for $|x|>R$.
 The
nonnegative parameter $\alpha$ represents
absorption. It is zero in the limit of zero
absorption, where $L^2$ solutions of \eqref{ho}
become solutions satisfying the standard
radiation condition.

The Helmholtz equation (\ref{ho}) is widely used
to model wave propagation problems in application
areas like electromagnetics, geophysics and
acoustics. Numerical simulation
of Helmholtz becomes expensive when the frequency
of the waves is high. In direct discretization
methods a large number of grid points is then
needed to resolve the wave oscillations, and the
computational cost to maintain constant accuracy
grows algebraically with the frequency. The
Helmholtz equation is typically even more
difficult to handle in this regime than
time-dependent wave equations, as numerical
discretizations lead to large indefinite and
ill-conditioned linear systems of equations, for
which it is difficult to find efficient
preconditioners \cite{Erlaga1}. At sufficiently
high frequencies direct simulations are not
feasible.


As an alternative
 one can use high frequency asymptotic models for wave propagation, such as geometrical optics \cite{Keller:62, EngRun:03, Runborg:07}, which is obtained when the frequency tends to infinity.
 The solution of the partial differential equation (PDE)
is assumed to be of the form
\begin{align}\label{GOform}
u  =
a e^{ik\phi},
\end{align}
where $\phi$ is the phase, and $a$ is the
amplitude of the solution. In the limit $k\to
\infty$ the phase and amplitude are independent
of the frequency and vary on a much coarser scale
than the full wave solution.
 They can therefore be computed at a computational cost independent of the frequency. However, a main drawback of geometrical optics is that the model breaks down at caustics, where rays concentrate and the predicted amplitude
 $a$ becomes unbounded.

Gaussian beams form another high frequency
asymptotic model which is closely related to
geometrical optics. However, unlike geometrical
optics, the phase $\phi$ is
complex-valued, and there is
no breakdown at caustics.
The solution is still assumed to be of
the form \eq{GOform}, but it is
concentrated near a single ray of geometrical
optics. To form such a solution, we first pick a
ray and solve systems of ordinary differential
equations  along it to find the Taylor expansions of the phase and
amplitude in variables transverse to the ray. Although the phase function is
real-valued along the central ray, its imaginary part  is chosen so that the solution
decays exponentially away from the central ray,
maintaining a Gaussian-shaped profile. For the simplest {first order}
beams the phase
$\phi$ is a second order Taylor expansion, while
the amplitude $a$ is a zeroth order expansion. For wave equations  one can use time
as a parameter for the rays, and
the expressions for the phase and amplitude are
\begin{align}\label{phaseamp}
\phi(t,y)=\phi_{0}(t) + (y-x(t))\cdot p(t) + \frac12(y-x(t))\cdot  M(t)(y-x(t)),
\qquad
a(t,y) = a_0(t)
\end{align}
where $x(t)$ is the geometrical optics ray,
$p(t)$ is the direction of the ray and the second
derivative matrix $M(t)$ encodes the width and
curvature of the beam; $M$ has a positive
definite imaginary part which ensures the beam
has a Gaussian shape. In the
Helmholtz case, since there is no longer a
distinguished variable with level sets transverse
to the rays, one uses Taylor expansion in the
plane orthogonal to the ray direction.
{Higher order beams are constructed through higher
order Taylor expansions in \eq{phaseamp}.}

The existence of Gaussian beam solutions to the
wave equation has been known since sometime in
the 1960's, first in connection with lasers, see
Babi\v{c} and Buldyrev \cite{BabicBuldyrev:1991}.
Later, they were used in the analysis of
propagation of singularities in PDEs by
H\"ormander \cite{Hormander:71} and Ralston
\cite{Ra82}. In the context of the Schr\"odinger
equation first order beams correspond to
classical coherent states. Higher order versions
of these have been introduced to approximate the
Schr\"odinger equation in quantum chemistry by
e.g.\mbox{} Heller \cite{Heller:81}, Hagedorn
\cite{Hagedorn:80}, Herman and Kluk
\cite{HermanKluk:84}.

More general high frequency solutions that are
not necessarily concentrated on a single ray can
be described by superpositions of Gaussian beams.
This idea was first introduced by Babi\v{c} and
Pankratova in \cite{Babic:1973} and was later
proposed as a method for approximating wave
propagation by Popov in \cite{Popov:1982}.
Letting the beam parameters depend on their
initial location $z$, such that $x=x(t;z)$,
$p=p(t;z)$ etc., and $a=a(t,y;z)$,
$\phi=\phi(t,y;z)$, the approximate
 solution for an initial value problem can be expressed with the
superposition integral
\begin{align}
   u(t,y) =\left(\frac{k}{2\pi}\right)^\frac{d}{2}
   \int_{K_0} a(t,y;z)  e^{ik\phi(t,y;z)} dz \ ,
\end{align}
where $K_0$ is a compact subset of $\Real^d$.

It should be mentioned that there are other
related Gaussian beam like approximations. In the
{\em thawed} Gaussian approximation
\cite{Heller:75} the phase $\phi$ is always a
second order polynomial. Higher order is obtained
by instead taking a higher order polynomial in
the amplitude, to correct also for errors in the
phase. {\em Frozen} Gaussian approximations
\cite{Heller:81,HermanKluk:84} also use a second
order polynomial for the phase $\phi$, but with a
fixed size of the second derivative
($M(t)$=constant). Single frozen Gaussians are
therefore {\em not} asymptotic solutions to the
wave equation. However, superpositions of frozen
Gaussians are and they can be thought of as an
efficient linear basis for the wave equation.

Numerical methods based on Gaussian beam
superpositions go back to the 1980's with work by
Popov, \cite{Popov:1982,Katchalov_Popov:1981},
Cerveny \cite{Cerveny_etal:1982} and Klime\v{s}
\cite{Klimes:1984} for high frequency waves and
e.g.\mbox{}  Heller, Herman, Kluk
\cite{Heller:81,HermanKluk:84} in quantum
chemistry. In the past decade there was a renewed
interest in such methods for waves following
their successful use in seismic imaging and oil
exploration by Hill \cite{Hill:1990,Hill:2001}.
Development of new beam based methods are now the
subject of intense interest in the numerical
analysis community and the methods are being
applied in a host of applications, from the
original geophysical applications to gravity
waves \cite{TQR:2007}, the semiclassical
Schr\"{o}dinger equation
\cite{FaouLubich:06,JinWuYang:08,LeungQian:09},
and acoustic waves \cite{Tanushev:08}. See also
the survey of Gaussian beam methods in
\cite{JinMarSpa:12}. Individual beams are
normally computed in a Lagrangian fashion by
solving ODEs along the central rays. The
superposition is then replaced by a discrete
summation of beams. There are also more recent
numerical techniques based on Eulerian
formulations of the problem \cite{LeuQiaBur:07,JinWuYang:08,JinEtal:10,LeungQian:09,QianYing:2010}.
In these methods a PDE is derived for the parameters in the beams, i.e.\mbox{} the quantities in the ODEs. This is coupled with a
level-set PDE for the ray dynamics. With the Eulerian formulation the result is no longer a superposition of asymptotic solutions to the wave equation! For superpositions over subdomains moving with the Hamiltonian flow,  it was shown directly  in \cite{LiuRalston:09, LiuRalston:10} that they are asymptotic solutions without reference to standard Gaussian beams. Numerical approaches for treating general high frequency initial data for superposition over physical space were considered in \cite{TET:2009,AETT:2010} for the wave equation.

In this paper we study the accuracy in terms of
$k$ of Gaussian beams and superpositions of
Gaussian beams for the Helmholtz equation
\eq{ho}. This would give a rigorous foundation
for beam based numerical methods used to solve
the Helmholtz equation in the high frequency
regime. In the time-dependent case several such
error estimates have been derived in recent
years: for the initial data \cite{Tanushev:08},
for scalar hyperbolic equations and the
Schr\"odinger equation
\cite{LiuRalston:09,LiuRalston:10,LRT11}, for
frozen Gaussians \cite{RousseSwart:09,LuYang:12}
and for the acoustic wave equation with
superpositions in phase space
\cite{BougachaEtal:09}. The general result is
that the error between the exact solution and the
Gaussian beam approximation decays as $k^{-N/2}$
for $N$-th order beams in the appropriate Sobolev
norm. There are, however, no rigorous error
estimates of this type available for the
Helmholtz equation. What is known is how well the
beams asymptotically satisfy the equation,
i.e.\mbox{} the size of $L_n u$ for a single
beam. Let us also mention an estimate of the
Taylor expansion error away from caustics,
\cite{MotamedRunborg:09}.

The analysis of Gaussian beam superpositions for
Helmholtz presents a few new challenges compared
to the time-dependent case. First, it must be
clarified precisely how beams are generated by
the source function and how the Gaussian beam
approximation is extended to infinity. This is
done in \S 2 and \S 3 for a compactly supported
source function that concentrates on a
co-dimension one manifold. Second, additional
assumptions on the index of refraction $n(x)$ are
needed to get a well-posed problem with
$k$-independent solution estimates and a
well-behaved Gaussian beam approximation at
infinity. The conditions we use are that $n(x)$
is non-trapping and that there is an $R$ for
which $n(x)$ is constant when $|x|>R$.

In \S 4 we  consider the difference between the
Gaussian beam approximation and the exact
solution to the radiation problem with the
corresponding source function. Here we are
interested in behavior of the local $L^2$ norm
$||u_{GB}-u||_{L^2(|x|<R)}$ as $k\to \infty$.
This depends on the well-posedness of the
radiation problem. There are a variety of
estimates that apply here
\cite{PerthameVega:99,CastellaJecko:06}, but the
Laplace-transform based estimates of Vainberg
\cite{V75,V89} suffice for our purposes. In \S 5
we compare the  Gaussian beam approximation with
the result of stationary phase expansion of the
exact solution in a simple example.

Sections \S 6 and \S 7 are devoted to
superpositions of beams with fundamental source
terms. Our main result is Theorem~\ref{ee} where
we are able to show that the error between
superposition of {$N$-th order beams and the exact
outgoing solution decays as $k^{-N/2}$}
independent of dimension and presence of
caustics. This is consistent with the optimal
results of \cite{LRT11} in the time-dependent
setting. Finally, \S 7 gives an example of how
beams can be constructed for more general source
functions.

\section{Construction of Gaussian beams}
In this section we construct the Gaussian
beam solutions for \eq{ho} when $f$ is compactly
supported on a co-dimension one manifold. 
This construction has become standard (see, for
example, \cite{Ra82} or \cite{KKL01}) and we
review some details here which
will be used later.
The
form of the beam solutions is
\begin{equation}\label{wkb}
u(x;k)=e^{ik\phi(x)}(a_0(x)+a_1(x)k^{-1}+\cdots+a_\ell(x)k^{-\ell}).
\end{equation}
Each beam concentrates on a geometrical optics
ray $\gamma=\{x(s):s\in \Bbb R\}$, which is the
spatial part of the bicharacteristics
$(x(s),p(s))$ defined by the flow for the
Hamiltonian $H(x,p)=|p|^2-n^2(x)$
\begin{equation} \label{bich}
\dot x= 2p, \quad \dot p =-\nabla_x n^2(x).
\end{equation}
We assume that there is a number $R>0$ such that
the (smooth) index of refraction satisfies
$n(x)\equiv 1$ when $|x|>R$ and that the source
function $f$ is compactly supported in $\{|x|<
R\}$. 
Here we also restrict the construction of the Gaussian
beam solution to the larger region
$|x|\leq 6R$.
The essential additional hypothesis for our
construction is that the index of refraction does
not lead to trapped rays. The precise
non-trapping condition is that there is an $L$
such that $|x(L)|>2R$ for all solutions with
$|x(0)|<R$ and $H(x(0),p(0))=0$.
Note that this implies that $|x(s)|>2R$ for $s>L$
since rays are straight lines when $n(x)\equiv 1$.

Applying $L_n$ in \eq{ho} to \eq{wkb}  we have
\begin{equation}\label{Lnu}
L_n u=
e^{ik\phi} \sum_{j=-2}^\ell c_j(x) k^{-j},
\end{equation}
where 
\begin{align*}
c_{-2} & = (n^2- |\nabla_x \phi|^2)a_0=_{def}E(x)a_0,\\
c_{-1} &=i \alpha n^2a_0 +\nabla_x \cdot( a_0\nabla_x\phi) +\nabla_x a_0 \cdot \nabla_x \phi+Ea_1,\\
c_j &=i \alpha n^2a_{j+1}  +\nabla_x\cdot( a_{j+1}\nabla_x \phi) +\nabla_x a_{j+1} \cdot \nabla_x \phi+Ea_{j+1}
 +\Delta_x a_j, \quad\ j=0, 1, \ldots ,\ell .
\end{align*}
ODEs for $S(s)=\phi(x(s))$, $M(s)=D^2\phi(x(s))$ and $A_0(s)=a_0(x(s))$ 
arise from requiring that
$c_{-2}$ vanishes to third order on the ray
$x(s)$,  and that $c_{-1}$ vanishes to first
order on the ray. 
It leads to the equations
\begin{align}\label{prop}
 \dot S=2n^2(x(s))\qquad
\dot M =D^2(n^2)(x(s))-2M^2\qquad
\dot A_0 = -{\rm tr}(M(s))A_0-\alpha n^2(x(s))A_0.
\end{align}
This amounts to
constructing a ``first order" beam.
Higher order beams can be constructed by
requiring $c_{-2}$ vanishes to higher order on
$\gamma$. Then one can require that the $c_j$'s
with {$j>-2$} also
 vanish to higher order, and obtain  a recursive set of linear equations for the partial derivatives of $a_0, a_1, \ldots,
a_\ell$. More precisely, for an $N$-th order beam $\ell = \lceil N/2\rceil-1$
in \eq{wkb}
and $c_j(x)$ should
vanish to order $N-2j-2$ when $-2\leq j\leq \ell-1$.

For initial data, we
let $S(0)=0$ and choose $M(0)$ so that
\begin{equation}\label{Mprop}
M(0)=M(0)^\top, \qquad M(0)\dot x(0)=\dot p(0), \qquad
\text{${\rm Im}\{M(0)\}$ is positive definite on $\dot x(0)^\perp$}.
\end{equation}
Then for all $s$ the matrix
$M(s)$ inherits the properties of $M(0)$:
 $ M(s)\dot x(s) = \dot
p(s)$, $M(s)=M(s)^\top$,
 and   ${\rm Im}\{M(s)\}$  is
positive definite on the orthogonal complement of
$\dot x(s)$, see \cite{Ra82}. For the amplitude
we take $A_0(0)=1$. We can solve the ODE for
$A_0$ explicitly, and obtain
$$
A_0(s)= \exp\left( -\int_0^s (\alpha n^2(x(\tau))+{\rm tr}(M))d\tau\right).
$$

The phase $\phi$ in \eq{wkb} can be any function satisfying
$\phi(x(s))=S(s)$, $\nabla\phi(x(s))=p(s)$ and
$D^2\phi(x(s))=M(s)$. However, to write down
such a function we need to have $s$ as a function
of $x$. Since we have $\dot x(s)\neq 0$, $x(s)$
traces a smooth curve $\gamma$ in $\R^d$, and the
non-trapping hypothesis implies that this curve
is a straight line when $|s|>L$. 
{We let
$$
 \Omega(\eta) = \{x\,:\, |x|\leq 6R\ \text{\rm and}\  |x-\gamma|\leq \eta\},
$$ 
be the tubular neighborhood of $\gamma$ with radius $\eta$
in the ball $\{|x|\leq 6R\}$.
By choosing $\eta$ small enough,  
we can uniquely define $s=s(x)$ for all $x\in\Omega(\eta)$
such that $x(s)$ is the closest point on $\gamma$ to $x$,
provided $\gamma$ has
no self-intersections.}
We then define the phase function $\phi$
and amplitude $A$
on $\Omega$ {for first order beams} by
\begin{equation}\label{ph}
\phi(x)=S(s) +p(s)\cdot (x-x(s))+\frac{1}{2}(x-x(s))\cdot M(s)(x-x(s)),\qquad
A(x)=A_0(s),
\end{equation}
with $s=s(x)$. 
Note that $s(x)$ is constant on planes orthogonal
to $\gamma$ intersected with {$\Omega(\eta)$}. Since
$\gamma$ can have only finitely many
self-intersections, we can cut $\gamma$ into
segments without self-intersections, and define
$s(x)$ on a tubular neighborhood each segment,
ignoring the endpoints. For this reason
self-intersections will not create difficulties,
and without loss of generality we will assume
that $\gamma$ has no self-intersections in what
follows. {The construction of the Gaussian beam
phase and amplitude for higher order beams
is carried out in a similar way \cite{Ra82}.}

\begin{figure}
\includegraphics[width=.6\textwidth]{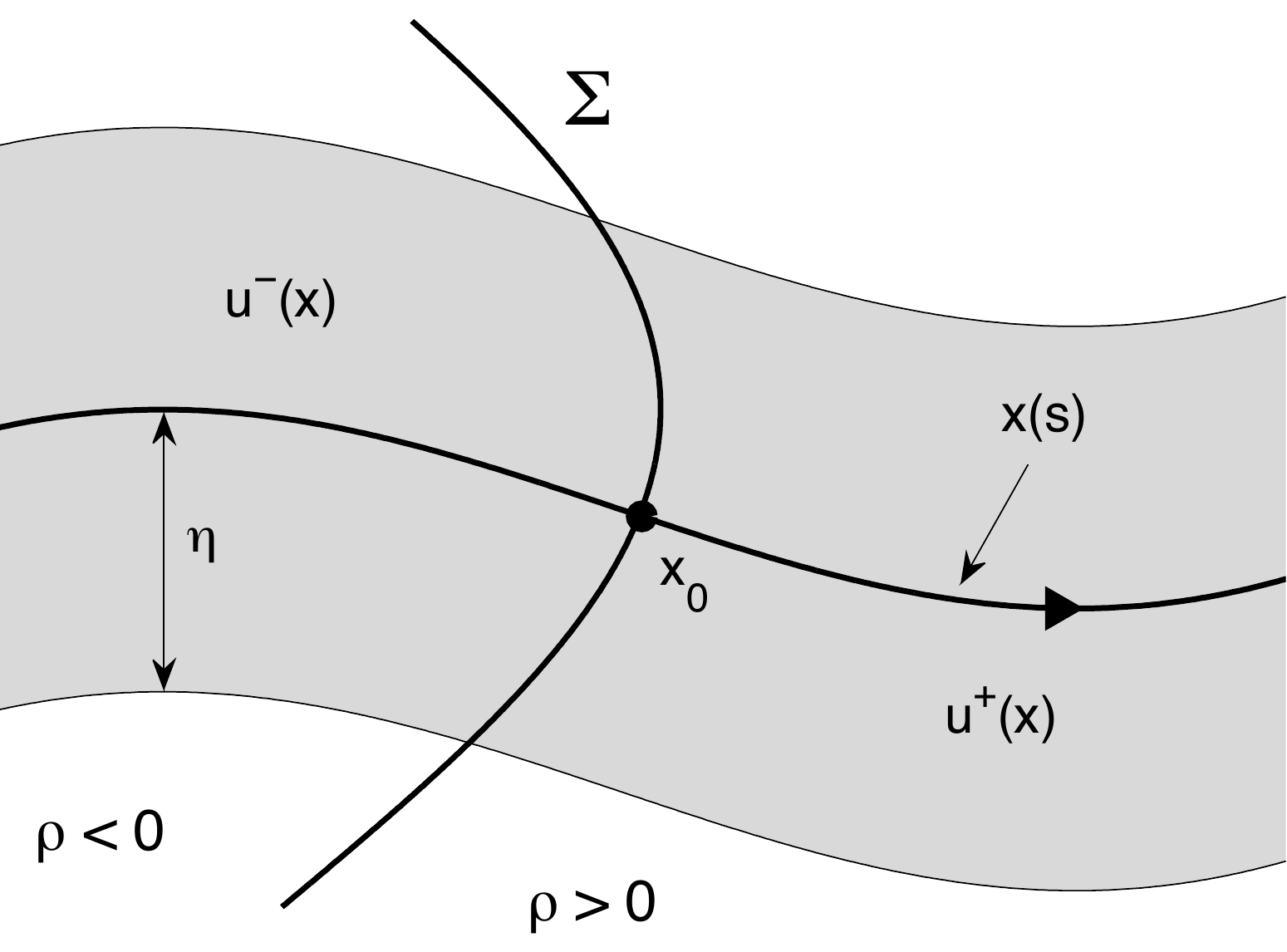}
\caption{Notation for the source in two dimensions.
The gray area indicates {$\Omega(\eta)$}.}
\label{fig:source}
\end{figure}

\subsection{Source}

To introduce the source functions that we will consider in this article let $\rho$ be a function such that $|\nabla \rho|=1$ on $\{x:\rho(x)=0\}$, and define $\Sigma$ to be the hypersurface $\{x:\rho(x)=0\}$. Given $x_0\in \Sigma$, we let $(x(s),p(s))$ be the solution of (6) with $(x(0),p(0))=(x_0,n(x_0)\nabla \rho(x_0))$. Since we assume no trapped rays and $n(x)\equiv 1$ when $|x|>R$, $x(s)$ and $p(s)$ are defined for $s\in \Bbb R$, and we set $\gamma =\{x(s), s\in \Bbb R\}$.  Then we can assume that $s(x)$ is defined on the tubular neighborhood {$\Omega(\eta)$} of $\gamma$ as above (assuming no self-intersections). We begin with a beam $u(x,k)$ concentrated on $\gamma$, and defined on {$\Omega(\eta)$}. If $u$ is first order, we can
define it by \eq{ph}. Then we define $u^+$ to be the restriction of $u$ to $\{x:\rho(x)\geq 0\}$. In order to have a source term which is a  multiple of $\delta(\rho)$, we need a second beam $u^-(x,k)$ defined on $\{x:\rho(x)\leq 0\}$ which is equal to $u^+$ on  $\Sigma$ for all $k$. Hence, writing $u^+(x,k)=A^+(x,k)e^{ik\phi^+(x)}$ and $u^-(x,k)=A^-(x,k)e^{ik\phi^-(x)}$, we must have $\phi^+=\phi^-$ and $A^+=A^-$ on $\Sigma$. Those requirements and $c_j=0,\ j=-2,\ldots,\ell$ at $x_0$ determine the Taylor series in the transverse variables at $x_0$ for $\phi^-$ and $A^-$. To see this suppose that $u^-$ is going to be a beam of order $N$ and that the coordinates on 
{$\Omega(\eta)$} are
 given by $(s,y)$ where $s=s(x)$ and $y=(y_1,\ldots,y_{d-1})$ is transversal. Then,
provided {$\eta$} is chosen small enough, $\Sigma$ is given by $s=\sigma(y)$ with $\sigma(0)=0$ and $\nabla \sigma(0)=0$. To determine the Taylor series in $y$ for $\phi^-(s,y)$ at $s=0$  one differentiates the equation $\phi^-(\sigma(y),y)=\phi^+(\sigma(y),y)$ with respect to $y$ and evaluates at
  $y=0$. When partial derivatives of $\phi^-$ with respect to $s$ appear in this calculation, they are determined by the requirement that $c_{-2}$ vanishes on $x(s)$ to order $N+2$. The Taylor series for $A^-$ in the transverse variables at $x_0$
is determined in the same way from $A^-(\sigma(y),y,k)=A^+(\sigma(y),y,k)$ for all $k$. To construct $u^-$, we use those Taylor series as  data at $s=0$ in solving the equations $c_j=0,\ j=-2,\ldots,\ell$ along $x(s)$.
 Since for an $N$-th order beam we only require that $c_j$ vanishes on $x(s)$ to order $N-2j-2$, we can still require that $\phi^+=\phi^-$ and $A^+=A^-$ {\it exactly} at points on $\Sigma$. Extending  $u^+$ to be zero in $\{x: \rho(x)<0\}$ and $u^-$ to be zero in $\{x: \rho(x)>0\}$, we define $u_{GB}=u^++u^-$. Then we have, setting $A=A^+=A^-$ on $\Sigma$,
 \begin{align}\label{source}
L_n u_{GB}&=\left[ik\left({\partial \phi^+\over \partial
\nu}-{\partial \phi^-\over \partial
\nu}\right)A+{\partial A^+\over \partial
\nu}-{\partial A^-\over \partial
\nu}
\right]e^{ik\phi^+}\delta(\rho)+f_{GB}
=_{def}g_0\delta(\rho)+f_{GB},
\end{align}
where $\nu(x)=\nabla \rho(x)$, the unit normal to $\Sigma$. We consider the singular part of $L_nu_{GB}$ in (11), i.e.  $g_0\delta(\rho)$, to be the source term and $f_{GB}$ to be the error from the Gaussian beam construction. Note that
\begin{equation}\label{fGBform}
f_{GB}=e^{ik\phi^+(x)}\sum_{j=-2}^{\ell}c^+_j(x)k^{-j}+e^{ik\phi^-(x)}\sum_{j=-2}^{\ell}c^-_{j}(x)k^{-j},
\end{equation}
where the $c_j^+(x)$ are extended to be zero when $\rho(x)<0$ and the $c_j^-(x)$ are extended to be zero when $\rho(x)>0$. For first order beams $\ell=0$ and \eq{prop} implies $c^\pm_{-2}(x)$ and $c^\pm_{-1}(x)$ are $O(|x-x(s(x))|^3)$ and $O(|x-x(s(x))|)$ respectively. Finally we restrict the support of $u_{GB}$ to {$\Omega(\eta)$} by multiplying it by a smooth cutoff  function supported in {$\Omega(\eta)$} which is identically one on {the smaller neighborhood $\Omega(\eta/2)$}. The cutoff function modifies $A^\pm$, and $f_{GB}$, outside {$\Omega(\eta/2)$}, but its contribution to \eq{source} is exponentially small in $k$ (see \cite{LRT11}), and we will disregard it from here on.

\subsection{Estimate of $f_{GB}$}

From the non-trapping condition, it follows that 
the length of a ray inside {$\Omega(\eta)$}
is bounded independently of starting point in $|x|\leq R$.
By construction, {$c_\ell^{\pm}(x)$ is bounded and}
\begin{equation}\label{cjform}
  c^\pm_{j}(x)= \sum_{|\beta|=N-2j-2} d^\pm_{\beta,j}(x)(x-x(s))^\beta,
  \qquad j=-2,\ldots,\ell-1,
\end{equation}
where $d^\pm_{\beta,j}(x)$ are bounded on {$\Omega(\eta)$}.
Hence, 
$$
|c^\pm_j(x)|\leq C_j|x-x(s)|^{N-2j-2},\qquad x\in {\Omega(\eta)}.
$$
Choosing  {$\eta$}  sufficiently
small, the construction also ensures that 
\begin{equation}\label{imphiest}
{\rm
Im}\{\phi^\pm\}(x) \geq  c |x-x(s)|^2,\qquad x\in {\Omega(\eta)},
\end{equation}
{see \cite{LRT11}.}
From the bound 
\begin{equation}\label{expest}
s^pe^{-as^2}\leq C_p a^{-p/2}e^{-as^2/2},
\qquad C_p = (p/e)^{p/2}, 
\end{equation}
with $p=N-2j-2$, $a=kc$ and $s=|x-x(s)|$ we then get
for $x\in{\Omega(\eta)}$,
\begin{align}
  |f_{GB}(x)|
  &\leq  e^{-k{\rm
Im}\{\phi^\pm\}(x)} \sum_{j=-2}^\ell |c^\pm_j(x)| k^{-j}
  \leq  e^{-kc|x-x(s)|^2 } \sum_{j=-2}^\ell 
  C_j|x-x(s)|^{N-2j-2} k^{-j}
\nonumber  \\
  &\leq  Ce^{-k{c}|x-x(s)|^2/2 }\sum_{j=-2}^\ell 
k^{-N/2+j+1} k^{-j}
\leq C e^{-kc|x-x(s)|^2/2 } k^{-N/2+1}.\label{fGBest}
\end{align}
We note that the constant is uniform in
$|x|\leq 6R$ and in particular for first order
beams $f_{GB}$ will be $O(k^{1/2}e^{-k\tilde{c}|x-x(s)|^2 })$.

\section{Extension of Gaussian beam solutions to infinity}

  In this section we extend
$u_{GB}(x)$ defined on $|x|\leq 6R$ to an outgoing
solution {$\tilde{u}_{GB}(x)$} in $\R^d$. For estimates on the validity
of the approximation it is essential to do this
so that $$
{\tilde{f}_{GB}=_{def} L_n \tilde{u}_{GB}-g_0\delta(\rho)},
$$
is supported in
$|x|<6R$ and 
is {$o(k)$}.

The main step in the extension is a simplified
version of the procedure used in \cite{MR78}.
Let
$G_{\lambda}(x)$ be the Green's function for the Helmholtz
operator $\Delta + \lambda^2$, where $\lambda$ may be complex valued.
When $\alpha\geq 0$, define
\begin{equation}\label{kalpha}
  k_\alpha := \sqrt{k^2+ik\alpha}.
\end{equation}
Then $L_1=\Delta  +i\alpha k + k^2=
\Delta  +k_\alpha^2$, and 
$G_{k_\alpha}$ is uniquely determined when $\alpha>0$ as the
inverse of the self-adjoint operator $L_1$;
for $\alpha=0$ it can be defined either as
$\lim_{\alpha\downarrow 0}G_{k_\alpha}$ or by
radiation conditions.
In the case $d=3$,
$$
G_{{k_\alpha}}(x)= -(4\pi)^{-1} \left(\frac{e^{ik_\alpha |x|}}{|x|} \right).
$$
To extend $u_{GB}$ we
introduce the cutoff function $\eta_{a}(x)$ in
$C^\infty(\R^d)$ with parameter $a\geq 1$:
\begin{align*}
\eta_a(x) =\left\{ \begin{array}{ll}
1 & |x|<(a-1)R \\0 &|x|>aR
\end{array}
\right.
\end{align*}
(see Figure~\ref{fig:beta_eta_cutoff}) and {define}
\begin{figure}
\includegraphics{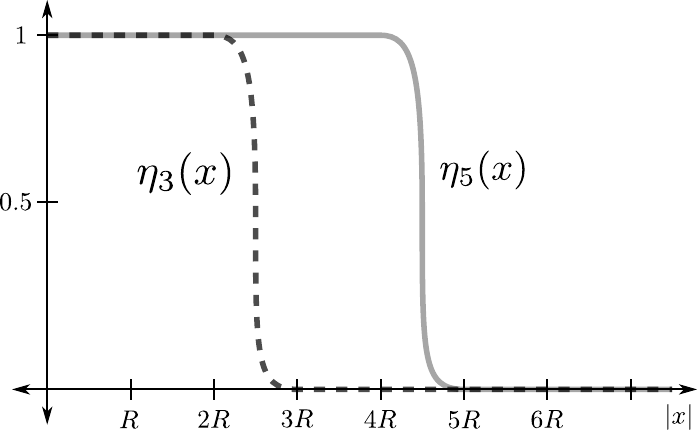}
\caption{The cut off functions $\eta_3(x)$ and $\eta_5(x)$.}
\label{fig:beta_eta_cutoff}
\end{figure}
\begin{equation}\label{h}
\tilde u_{GB} =\eta_3(x) u_{GB}(x) +\int G_{{k_\alpha}}(x-y)\eta_5(y) L_n[(1-\eta_3(y))u_{GB}(y)]dy.
\end{equation}
We also assume that $R$ is chosen large enough such that
the support of $g_0\delta(\rho)\subset \Sigma\cap{\Omega(\eta)}$ 
is inside $\{|x|<R\}$.

Consider {first} $L_n\tilde u_{GB}$ in the region
$\{|x|\geq R\}$. Since $L_n=L_1$ {as well as $g_0\delta(\rho)=0$} in this region
and $\eta_5\equiv 1$ on the support of $\eta_3$,
\begin{align*}
{\tilde{f}_{GB}(x)}&=L_n\tilde u_{GB}(x)=\eta_5(x) L_n[\eta_3(x)
u_{GB}(x)]+\eta_5(x)L_n[(1-\eta_3(x))u_{GB}(x)]\\
&=\eta_5(x)L_n u_{GB}
{=\eta_5(x){f}_{GB}(x)}.
\end{align*}
Since $\eta_5$ is supported on $|x|<5R$,
it follows that
$\tilde{f}_{GB}$ vanishes for $|x|>5R$.

{Consider next the region$\{|x|\leq R\}$
and let $v=\tilde u_{GB}-\eta_3 u_{GB}$, i.e.
the integral term in (\ref{h}).
Since, $\eta_3=1$ on $|x|<R$, we have in this region
$$
   \tilde{u}_{GB}-u_{GB} = v,\qquad
   \tilde{f}_{GB}-f_{GB} = L_nv.
$$
}
In view of the
estimate of {$f_{GB}$} it now suffices to show
that for $|x|\leq R$, $\partial_x^\beta v$ decays
rapidly when $k\to \infty$, for all
multi-indices, $|\beta|\leq 2$.

By the definition of the two cut-off functions,  we
have for $|x|\leq R$

\begin{align*}
v(x) &=\int_{\R^d}  G_{{k_\alpha}}(x-y)\eta_5(y) L_n[(1-\eta_3(y))u_{GB}(y)]dy\\
 & = \int_{2R\leq |y|\leq 5R} G_{k_\alpha}(x-y)\eta_5(y) L_1[(1-\eta_3(y))u_{GB}(y)]dy  .
\end{align*}

The fundamental solution  $G_{k_\alpha}$ has the form
$$G_{{k_\alpha}}(x)=
{e^{i{k_\alpha}|x|}\over|x|^{(d-1)/2}}w(x;{k_\alpha}),$$ where $w$
and its derivatives in $x$ are bounded by
{$|k_\alpha|^{\frac{d-3}{2}}\leq C k^{\frac{d-3}{2}}$} 
on compact subsets of $|x|\geq R
$, see Appendix.
Since $n(x)\equiv 1$ for $|x|>R$, in that
region $x(s)$ is a straight line and
$\nabla_x\phi^\pm(x(s))$ is a constant unit
vector. Since $x(s)$ is going out of $|x|\leq R$
when it crosses $|x|=R$,  at $x(s)=y$ with
$|y|\geq 2R$ the phases in $u_{GB}$ satisfy
$\nabla_x\phi^\pm(x(s))\cdot y\geq \cos (\pi/6)|y|$.
 Likewise when $|x|\leq R$
and $|y|\geq 2R$, $(y-x)\cdot y\geq
|y||y-x|\cos(\pi/6)$ (see
Figure~\ref{fig:exit-angle}). The form of
$u_{GB}$ (see (\ref{wkb})) gives the integrand in
(\ref{h}) the form $e^{ik\psi}b(y,k)$ with
$\psi(y)=\phi^\pm(y)+{(k_\alpha/k)}|x-y|$ and $b$ smooth in $y$,
{bounded together with its derivatives by $Ck^{\frac{d-3}{2}}$.}
Note that
$$
\nabla_y \psi= \frac{k_\alpha}{k} \frac{y-x}{|y-x|}+\nabla_y \phi^\pm.
$$
The preceding remarks show that, when $|x|\leq R$
and $k$ large, $\nabla_y \psi$ does not vanish on
the support of the integrand in (\ref{h}). Hence we can
use the identity
 $$e^{ik\psi}={\overline {{\nabla}_y\psi}\over ik |{\nabla}_y \psi|^2}\cdot{\nabla}_y(e^{ik\psi})$$
  and integrate by parts to  show that $v$ and its derivatives are order 
  {$k^{-m}$ for any $m$}.

This completes the verification of the
extension. We have shown that
\begin{equation}\label{ftildeGBest}
  \tilde{f}_{GB}(x) = {\eta_5(x)[f_{GB}(x)+r(x)]},
  \qquad 
||r||_{L^2(|x|<5R)} = O(k^{-m}).
\end{equation}
Hence, the size of $\tilde{f}_{GB}$ is of the same order as the size of $f_{GB}$, which is 
$O({k^{-N/2+1}}e^{-k\tilde{c}|x-x(s)|^2 })$.
Moreover,
\begin{equation}\label{uest}
||u_{GB}-\tilde{u}_{GB}||_{L^2(|x|<R)} = {O(k^{-m}),}
\end{equation}
{for any $m$.}
Note that, since $\tilde u_{GB}$ is
represented by $G_{\alpha,k}$ for $|x|$ large, it
is square-integrable ($\alpha>0$) or outgoing
($\alpha=0$).

\begin{figure}
\includegraphics{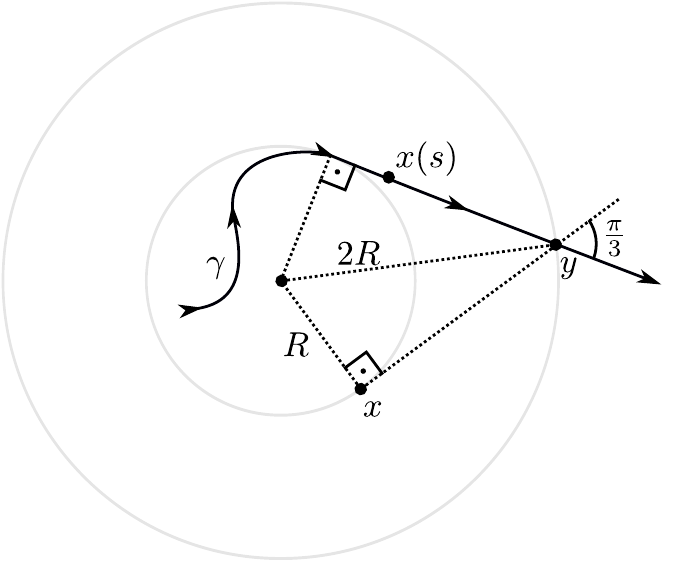}
\caption{Maximum angle.}
\label{fig:exit-angle}
\end{figure}


\section{The Error Estimate for $u_{GB}$}
In this section we will use an estimate showing
that the radiation problem is well-posed due to
Vainberg \cite{V75} and \cite{V89}. This will
give estimates on the accuracy of $u_{GB}$ as an
approximation to the exact solution in the region
$|x|\leq R$. Vainberg starts with the initial
value problem for wave equation in $\Bbb
R^d_x\times \Bbb R_t$
$$v_{tt}- n^{-2}\Delta
v=0, \quad v(0)=0, \quad v_t(0)=-n^{-2}g$$ and takes the
Fourier-Laplace transform
\begin{equation}\label{fl}
u(x,k)=\int_0^\infty e^{i\lambda t}v(t,x)dt
\end{equation}
to get the solution to
$$\Delta u + \lambda^2n^2u=g$$
satisfying radiation conditions. Taking advantage
of finite propagation speed, and the propagation
of singularities to infinity, he can estimate $u$
on bounded regions from the integral
representation (\ref{fl}), when $g$ has bounded
support {and the nontrapping condition holds}. In the notation of \cite{V75},
$u=[{{\mathcal R}}_\lambda](n^{-2}g)$,  where ${{\mathcal R}}_\lambda$ is
the operator
$${{\mathcal R}}_\lambda= (\lambda^2 +n^{-2}\Delta)^{-1}.$$
This is defined  for complex $\lambda$ as the
analytic continuation of  ${{\mathcal R}}_\lambda$ restricted
to the space $H^m_a$ with range in the space
$H^m(|x|<b)$.  The estimates take the following
form: there are constants $C$ and $T$ such that
\begin{equation}\label{est}
||{{\mathcal R}}_\lambda g||_{m+2-j,{(b)}}\leq C|\lambda|^{1-j}e^{T|{\rm Im
}\, \lambda |}||g||_{m,a},\ 0\leq j\leq 3.
\end{equation}
Here the norms are standard Sobolev norms on
$H^m_a(\Bbb R^d)$, the closure of
$C^\infty_c(|x|<a)$ in $||\cdot ||_m$, and
$H^m(|x|<b)$. One can assume that $b<a$. The
admissible set of $\lambda$ here is the set
$$U_{c_1, c_2}=\{\lambda \in \Bbb C: |\hbox{Im
}\lambda|<c_1 \log |\hbox{Re }\lambda|-c_2\}$$
for some $c_1, c_2>0$.   If $d$ is even, then one
has to add the condition
$$
-\pi/2 <\text{arg} \, \lambda <3\pi/2.
$$
This is Theorem 3 for $d$ odd and Theorem 4 for
$d$ even in \cite{V75}.

Here we will apply (\ref{est}) with
$g=n^{-2}\tilde f_{GB}$, $a=6R$, $b=R$ and
$\lambda{=k_\alpha} \in \Bbb C$
{with $k_\alpha$ defined in (\ref{kalpha}).} 
This makes $n^2{{\mathcal R}}_{k_\alpha}
g=\tilde u_{GB}-u_E$, where $u_E$ is the exact
solution to the radiation problem 
{(\ref{ho}) with $f=g_0
\delta(\rho)$ defined in (\ref{source}).}
Taking $m=0$
and $j=2$, we have
\begin{equation}\label{E-}
||{\tilde{u}_{GB}}-u_E||_{L^2(|x|<R)}\leq C|{k_\alpha}|^{{-1}}e^{T|{\small\rm Im}\;
k_\alpha |} ||\tilde
f_{GB}||_{L^2}.
\end{equation}
Note that 
$|k_\alpha|= k\left( 1+(\alpha/k)^2\right)^{1/4}$
and
$$
 |\text{Im}\; k_\alpha|=\frac{\alpha}{\sqrt{2}}  \left( \left(1+(\alpha/k)^2\right)^{1/2}-1\right)^{1/2},  \qquad  | \text{Re}\; k_\alpha|=\frac{k}{\sqrt{2}}
  \left( \left(1+(\alpha/k)^2\right)^{1/2}+1\right)^{1/2}.
$$
Hence $|\text{Im}\; k_\alpha|\leq C$, ${k_\alpha}\in U_{c_1,c_2}$ for some
$c_1,c_2>0$ and $|k_\alpha|>k$, so
\begin{equation}\label{E}
||u_{GB}-u_E||_{L^2(|x|<R)}\leq 
C|k|^{-1} ||\tilde
f_{GB}||_{L^2}+||\tilde{u}_{GB}-u_{GB}||_{L^2(|x|<R)},
\end{equation}
uniformly in terms of $\alpha$.  The estimates in
{\eq{ftildeGBest} and \eq{uest}}
ensure that
\begin{equation}\label{rayest}
||u_{GB}-u_E||_{L^2(|x|<R)}\leq
C|k|^{-1} ||
f_{GB}||_{L^2(|x|<5R)}.
\end{equation}
We observe here that
since \eq{ftildeGBest} and \eq{uest} hold uniformly for all 
beam starting points
$x_0\in\Sigma$ the estimate \eq{rayest} will also hold for
linear superpositions of beams, which we will discuss further below,
see \eq{utotest}.
Moreover, from \eq{fGBest} and the estimate \eq{degat} derived below,
we obtain
\begin{align*}
 ||f_{GB}||^2_{L^2(|x|<5R)}&\leq C k^{-N+2} \int_{{\Omega(\eta)}}
 e^{-2k\tilde{c}|x-x(s)|^2 } dx
 \leq C k^{-N+2+(1-d)/2}.
\end{align*}
This finally shows that for a single beam $u_{GB}$,
$$
||u_{GB}-u_E||_{L^2(|x|<R)}\leq C k^{-N/2-\sigma_d},\qquad \sigma_d = \frac{d-1}{4}.
$$
Note that the factor $k^{-\sigma_d}$ corresponds to the size of the $L^2$ norm of the
beam itself in $d$ dimensions, $||u_{GB}||_{L^2(|x|<R)}\sim k^{-\sigma_d}$, showing that the relative
error of the beam is bounded by $k^{-N/2}$.

\section{An  Example}
Using the notation
$x=(x_1,x^\prime)=(x_1,x_2,x_3)$, the  outgoing
solution to $$\Delta u+k^2u=2 ik
e^{-k|x^\prime|^2/2}\delta(x_1)$$ is given by
\begin{equation}\label{ii}
u(x,k)={-2i k\over 4\pi}\int_{\Bbb R^2}{e^{ik|x-(0,y^\prime)|-k|y^\prime|^2/2}\over |x-(0,y^\prime)|}dy^\prime.
\end{equation}
In this section we compare the approximation that
one gets by using the method of stationary phase
on this integral to the approximation given by
$u_{GB}$. The stationary phase approximation is
not uniform in $x^\prime$, and for $x^\prime\neq
0$ it simply gives $u(x_1,x^\prime,k)=O(k^{-N})$
for all $N$. However, when $x^\prime=0$, it gives
$u_{GB}(x_1,0)$.

The procedure for constructing $u^+$ given
earlier with  the source
$2ike^{-|x^\prime|^2/2}\delta(x_1)$, gives
$x(s)=(2s,0,0)$, $p(s)=(1,0,0)$, $S(s)=2s$,
$M(s)= {i\over 1+2is}P$ and $A(s)=(1+2is)^{-1}$,
where $P$ is the orthogonal projection on $\hat
e_1^\perp$. For $u^-$ one gets the same
results with $s$ replaced by $-s$ and $p(s)$
replaced by $-p(s)$. The definition of  $s(x)$
gives $s(x)=|x_1|/2$, and we have
\begin{equation}\label{gb}
u_{GB}(x,k)=(1+i|x_1|)^{-1}e^{ik \phi},\hbox{ where }\phi= |x_1|+{i\over2(1+i|x_1|)}|y^\prime|^2.
\end{equation}
\vskip.1in To apply stationary phase to
\eqref{ii} assume that $x_1\neq 0$. Then the
phase
 is given by $\psi(x,y^\prime)=|x-(0,y^\prime)|+i|y^\prime|^2/2$ and
 $$\psi_{y^\prime}= {y^\prime-x^\prime\over| x-(0,y^\prime)|}+iy^\prime.$$
 That vanishes and is real only when $y^\prime=x^\prime=0$. Then one has
 $$\psi_{y^\prime y^\prime}|_{x^\prime=y^\prime=0}=\left({1\over |x_1|}+i\right)I_{2\times 2}.$$
The stationary phase lemma (\cite{Hormander})
gives
\begin{equation}\label{sp}
u(x_1,0)={2\pi\over k}(\hbox{det}(-i\psi_{y^\prime y^\prime}(x_1)))^{-1/2}\left({-2i k\over 4\pi}{e^{ik|x_1|}\over |x_1|} + O(1)\right).
\end{equation}
Since
$$\hbox{det}(-i\psi_{y^\prime y^\prime}(x_1))=\left({-i\over |x_1|}+1\right)^2,$$
{and the choice of square root leads to
$$\left(\left({-i\over |x_1|}+1\right)^2\right)^{-1/2}=\left({-i\over |x_1|}+1\right)^{-1},$$
}one sees that the leading term in
\eqref{sp} is exactly \eqref{gb}.

\section{ Error Estimates for Superpositions}
Given a point $z \in\Sigma$, we relabel the
primitive source term $g_0$ in {\eq{source}} as
\begin{equation}\label{S}
g(x,z,k)=[ik\zeta_1(x)+\zeta_2(x)]e^{-k|x-z|^2/2}\delta(\rho),
\end{equation}
where $\zeta_j \in C_c^\infty$ and {$\zeta_1(x)=1$ on a
neighborhood of $x=z$}. Denoting the resulting beam
as  $u_{GB}(x; z)$, the error estimate \eqref{E}
is uniform in $z$ as long as $z$ remains in a
compact subset of $|x|<R$, for instance $|z|\leq
R/2$. If we let $z$ range over $\Sigma$, we can
form
\begin{equation}\label{source2}
g(x,k)\delta(\rho)= \left(\frac{k}{2\pi}\right)^{(d-1)/2} \int_{\Sigma}g(x,z, k)h(z)dA_z,
\end{equation}
and
\begin{equation}\label{super}
u(x)=\left(\frac{k}{2\pi}\right)^{(d-1)/2} \int_\Sigma u_{GB}(x; z)h(z)dA_z
\end{equation}
is an approximation to the exact solution for the
source $g(x,k)\delta(\rho)$ satisfying the
estimate \eqref{E}.

We now state the main result of error estimates
for superposition (\ref{super}).
\begin{thm}\label{ee} Assume that  $n(x)$ is smooth, non-trapping, positive and equal to 1 when $|x|>R$.    Let $u_E$ be the exact solution to (\ref{ho}) with the source $f=g(x, k)\delta(\rho)$ in \eq{source2}, and $u$ be the Gaussian beam superposition defined in  (\ref{super})
based on {$N$-th order beams}.  We then have the following estimate
\begin{equation}\label{result}
\|u-u_E\|_{L^2(|x|\leq R)} \leq {Ck^{-N/2},}
\end{equation}
where $C$ is independent of $k$ but may
depend on $R$.
\end{thm}
In order to simplify the notation,
we specify
$\rho(x)=x_1$ and  $y=(0, z)$ for $z\in \Sigma
\subset \mathbb{R}^{d-1}$.  The superposition
thus can be written as
\begin{equation}\label{super+}
u_{}(x)=\left(\frac{k}{2\pi}\right)^{(d-1)/2} \int_{\Sigma} u_{GB}(x; z)h(z)dz,
\end{equation}
and the residual
\begin{equation}\label{fres}
L_n{u}-L_n u_{E} =
   f(x) =\left(\frac{k}{2\pi}\right)^{(d-1)/2} \int_{\Sigma} f_{GB}(x; z)h(z)dz.
\end{equation}
By the definition of $u_E$ and the source $g(x, k)\delta(\rho)$,
the residual
$f$ contains only regular terms. 
We can therefore extend the superposition $u$ to $\tilde{u}$ 
in the same way as in \S 3, and define
$\tilde{f}=L_n \tilde{u}-L_n u_{E}$.
As observed above, \eq{ftildeGBest} and \eq{uest} hold uniformly for all $z\in\Sigma$, and
the same steps as in \S 4 therefore lead to an estimate corresponding to \eq{rayest},
namely
\begin{equation}\label{utotest}
||u - u_E||_{L^2(|x|\leq R)}\leq Ck^{-1}||
f ||_{L^2(|x|<5R)}.
\end{equation}
{We let $x(s;z)$ be the ray originating in $z$, $x(0,z)=z$ and
we denote by
$\Omega(\eta;z)$ the corresponding  tubular neighborhood of
radius $\eta$, in the ball $\{|x|\leq 5R\}$.}
By choosing $\eta>0$ sufficiently small, we can
thus ensure that  $s=s(x;z)$ is well defined {on $\Omega(\eta;z)$}.  
In what  follows we denote  $x(s(x, z);
z)$ by $\gamma$ or $\gamma(x; z)$.  
Moreover, we introduce the cutoff function $\varrho_\eta(x)\in C^\infty(\Real^d)$ as
 \begin{align}\label{cutoff}
\varrho_\eta(x) \geq 0 \quad \mbox{ and } \quad \varrho_\eta(x)= \left\{ \begin{array}{ll} 1 \mbox{ for } |x|\leq \eta/2, \\
0 \mbox{ for } |x|\geq \eta, \end{array}
\right.
\end{align}
{such that $\varrho_\eta(x-\gamma(x; z))$ is supported on
$\Omega(\eta;z)$ and is identically one on $\Omega(\eta/2;z)$.}
The form \eq{fGBform} of $f_{GB}(x;z)$ will then be
\begin{align*}
f_{GB}(x;z)
&=\left(e^{ik\phi^+(x;z)}{\sum_{j=-2}^{\ell}}c^+_j(x;z)k^{-j}+e^{ik\phi^-(x;z)}{\sum_{j=-2}^{\ell}}c^-_{j}(x;z)k^{-j}\right)
\varrho_\eta \left(x-\gamma\right) + O(k^{-\infty})\\
&=\sum_{\alpha} k^{j_\alpha}e^{ik\phi_\alpha(x;z)}d_\alpha(x;z)(x-\gamma)^{\beta_\alpha}
\varrho_\eta \left(x-\gamma\right) + O(k^{-\infty}),
\qquad
\end{align*}
with bounds
$$
{|\beta_\alpha|\leq N+2,\qquad
2j_\alpha\leq {2-N+|\beta_\alpha|}.}
$$
The sum over $\alpha$ is finite,
$d_\alpha$ involves the functions $d_{\beta,j}^\pm$ in \eq{cjform}
and $\phi_\alpha$ is either $\phi^+$ or $\phi^-$.
Moreover, $O(k^{-\infty})$ indicates terms exponentially small
in $1/k$.
After neglecting these terms and
using \eq{fres} it follows that
we can bound the $L^2$ norm of $f$ by
\begin{align*} 
\|f\|^2_{L^2(|x|\leq 5R)}
&\leq Ck^{d-1}
\sum_{\alpha}
\left\|
\int_{\Sigma}{k^{\frac{2-N+|\beta_\alpha|}2}}
e^{ik\phi_\alpha}d_\alpha(x-\gamma)^{\beta_\alpha}
\varrho_\eta h 
dz\right\|^2_{L^2(|x|\leq 5R)}\\
&= {Ck^{d-N}}\sum_\alpha
\int_{|x|\leq 5R} \int_{\Sigma}\int_{\Sigma} I_\alpha(x,z, z')dzdz'dx,
\end{align*}
where the terms $I_\alpha$ are of
the form
\begin{align*}
  I_\alpha(x,z, z') &= k^{1+|\beta|} e^{ik \psi(x,z,z')}
       g(x; z')\overline{g(x; z)} \\
    &\qquad  \times  \left(x-\gamma \right)^\beta \left(x- \gamma' \right)^\beta
       \varrho_\eta\left(x-\gamma\right)\varrho_\eta \left(x-\gamma' \right), \quad |\beta|\leq 3.
\end{align*}
Here $g(x;z)=d_\alpha(x;z)h(z)$
and
\begin{align}\label{ps}
\psi(x, z, z'):= \phi(x; z')-\overline{\phi(x; z)},
\end{align}
with $\phi$ being either of $\phi^{\pm}$.
The function $g$ and its derivatives are bounded,
\begin{equation}\label{g}
\sup_{z\in \Sigma, x\in {\Omega(\eta;z)}}|\partial_x^\lambda g(x;z)|\leq C_\lambda,
\end{equation}
for any $|\lambda| \geq 0$.

Let $\chi_j(x,z,z')\in C^{\infty}$ be a partition
of unity such that
$$
  \chi_1(x,z,z') = \begin{cases} 1, & {\rm when}\ |\gamma(x,z)-\gamma(x,z')|> \theta |z-z'|,\\
 0, & {\rm when}\ |\gamma(x,z)-\gamma(x,z')|< \frac12\theta |z-z'|
  \end{cases}
$$
and $\chi_1+\chi_2 = 1$. Moreover, let
$$
I_1=\chi_1(x,z,z')I_\alpha(x,z, z') ,\qquad
I_2=\chi_2(x,z,z')I_\alpha(x,z, z') ,
$$
so that $I_\alpha(x,z, z')=I_1+I_2$.

The rest of this section is dedicated to establishing the
following {inequality}
\begin{equation}\label{Izz}
\left| \int_{|x|\leq 5R}\int_{\Sigma}\int_{\Sigma} I_j(x,z, z')dxdzdz' \right|\leq Ck^{2-d}
\end{equation}
for $j=1,2$. With this estimate  we  have 
$\|f\|_{L^2(|x|\leq 5R)}\leq {Ck^{1-N/2}}$, which together
with \eq{utotest}
lead to the desired estimate (\ref{result}).

A key ingredient in establishing estimate
(\ref{Izz}) is a slight generalization of
the non-squeezing lemma obtained
in \cite{LRT11}.  It says that the distance in phase space between
two smooth Hamiltonian trajectories at two parameter values
$s$ that depends smoothly on the initial position $z$,
will not shrink from its initial
distance, even in the presence of caustics.
The lemma is as follows:

\begin{lem}[Non-squeezing lemma]\label{nonsqueezing}
Let  $X=(x(s; z), p(s; z))$ be the
bi-characteristics starting from $z \in \Sigma$ {with $\Sigma$ bounded}. Assume that $p(0;
z)\in C^2(\Sigma)$
{is perpendicular to $\Sigma$ for all $z$}, that $|p(0;z)|=n(z)$ and
that $\inf_{z} n(z)=
n_0 >0$. 
Let $S(z)$ be a Lipschitz continuous function
on $\Sigma$ with Lipschitz constant $S_0$.
Then, there exist
positive constants $c_1$ and $c_2$ depending on
$L$, $S_0$ and $n_0$, such that
\begin{align}\label{NonSqueezeIneq}
   c_1 |z-z'| \leq |p(S(z); z)-p(S(z'); z')|+|x(S(z);z)-x(S(z');z')|\leq c_2 |z-z'| \ ,
\end{align}
for all $z,z'\in \Sigma $ and $|S(z)|, |S(z')|\leq L$.
\end{lem}
\begin{proof}
With the assumptions given here,  
the non-squeezing lemma proved in \cite{LRT11} states that
there are positive constants $0<d_1\leq d_2$ such that
\begin{equation}\label{origns}
  d_1 |z-z'| \leq |p(s; z)-p(s; z')|+|x(s;z)-x(s;z')|\leq d_2 |z-z'| \ ,
\end{equation}
for all $z,z'\in \Sigma $ and $|s|\leq L$, i.e.\mbox{} essentially the case
$S(z)\equiv$ constant.
Since the Hamiltonian for the flow \eq{bich} is regular for all $p$, $x$, and
the initial data $p(0;z)$ is $C^2(\Sigma)$, the derivatives $\partial^{\alpha}_{s,z}x$
and $\partial^{\alpha}_{s,z}p$ with $|\alpha|\leq 2$ 
are all bounded on $[-L,L]\times \Sigma$ by a constant $M$.
Then, for the right inequality
in \eq{NonSqueezeIneq}, we have
\begin{align*}
\lefteqn{
|p(S(z); z)-p(S(z'); z')|+|x(S(z);z)-x(S(z');z')|} \hskip 15 mm & \\
\leq{}& 
|p(S(z); z)-p(S(z'); z)|+|p(S(z');z)-p(S(z');z')|\\
& +
|x(S(z); z)-x(S(z'); z)|+|x(S(z');z)-x(S(z');z')|\\
\leq{}& 
2M |S(z)-S(z')| +
d_2|z-z'| \leq 
(2MS_0 +
d_2)|z-z'| =: c_2 |z-z'|,
\end{align*}
by \eq{origns} and the Lipschitz continuity of $S(z)$.
For the left inequality in \eq{NonSqueezeIneq},
\begin{align}\label{nsest1}
\lefteqn{  |x(S(z);z)-x(S(z');z')|+|p(S(z);z)-p(S(z');z')| }
\hskip 15 mm &\\
 \geq{}&
  |p(S(z);z)-p(S(z);z')|-
  |p(S(z);z')-p(S(z');z')|\notag\\
& + |x(S(z);z)-x(S(z);z')|-
  |x(S(z);z')-x(S(z');z')|\notag\\
  \geq{}&
  d_1 |z-z'| -
  |p(S(z);z')-p(S(z');z')|-
  |x(S(z);z')-x(S(z');z')|\notag\\
  \geq{}&
  d_1 |z-z'| - 2M|S(z)-S(z')|,\notag
\end{align}
where we again used \eq{origns}.
Next we will estimate $|S(z)-S(z')|$ using
$|x(S(z');z')-x(S(z);z)|$.
From Taylor expansion of $x$
around $z$, and the fact that $x_s=2p$, we have
$$
  x(S(z');z')-x(S(z);z) =
   2p(S(z);z)(S(z)-S(z'))
   +D_z x(S(z);z)(z'-z) + R(z,z'),
$$
where 
\begin{equation}\label{nsest2}
|R(z,z')|\leq M\left(|S(z)-S(z')|^2 + |z-z'|^2\right)
\leq M(1+S_0^2)|z-z'|^2.
\end{equation}
Moreover,
\begin{align*}
  \frac{d}{ds} p(s;z)^TD_z x(s;z) &= 
  p_s(s;z)^TD_z x(s;z) +
  p(s;z)^TD_z x_s(s;z)\\
  &= 
  -\nabla_x n^2(x(s;z))^TD_z x(s;z) +
  2p(s;z)^TD_z p(s;z) \\
  &= \nabla_z H(x(s;z),p(s;z))
  = \nabla_z H(x(0;z),p(0;z)) = 0,
\end{align*}
by the choice of data at $s=0$.
Therefore, since $p(0;z)$ is orthogonal to $\Sigma$ and $x_{z_j}(0;z)$ are
tangent vectors to $\Sigma$, we have $p(s;z)^TD_z x(s;z)=0$ for all $s$
and
\begin{equation}\label{nsest3}
  |x(S(z);z)-x(S(z');z')|  \geq 
   2|p(S(z);z)||S(z)-S(z')| - |R|\geq 
   2n_0|S(z)-S(z')| - |R|.
\end{equation}
Together \eq{nsest1}, \eq{nsest2} and \eq{nsest3} now give
\begin{align*}
\lefteqn{  |x(S(z);z)-x(S(z');z')|+|p(S(z);z)-p(S(z');z')| }
\hskip 15 mm & \\
&\geq d_1 |z-z'| - \frac{M}{n_0}|x(S(z);z)-x(S(z');z')|
-\frac{M^2(1+S_0^2)}{n_0}|z-z'|^2,
\end{align*}
which implies
$$
|x(S(z);z)-x(S(z');z')|+|p(S(z);z)-p(S(z');z')| 
\geq \tilde{d}_1|z-z'|\left(1 - m |z-z'|\right),
$$
with $m = M^2(1+S_0^2)/(n_0d_1)$ and $\tilde{d}_1=d_1/(1+M/n_0)$.
The lemma is thus proved for $|z-z'|\leq 1/2m$ 
with $c_1=\tilde{d}_1/2$.
On the other hand, if $|z-z'|\geq 1/2m$ there is a
number $c(m)$ such that
$$
  \mathop{\inf_{z,z'\in \Sigma,\ |z-z'|\geq 1/2m}}_{|s|\leq L,\ |s'|\leq L} 
  |p(s; z)-p(s'; z')|+|x(s;z)-x(s';z')| =: c(m)>0,
$$
by the uniqueness of solutions to the Hamiltonian system.
Hence, in particular, for $|z-z'|\geq 1/2m$,
$$
|x(S(z);z)-x(S(z');z')|+|p(S(z);z)-p(S(z');z')|
\geq c(m) \geq \frac{c(m)}{\Lambda} |z-z'|,
$$
where $\Lambda = \sup_{z,z'\in\Sigma} |z-z'|<\infty$ is the
diameter of the bounded set $\Sigma$.
This proves the lemma with $c_1 = \min(\tilde{d}_1/2, c(m)/\Lambda)$.
\end{proof}

We now prepare  some main  estimates  for proving
(\ref{Izz}).

\begin{lem}[Phase estimates]\label{PhaseEst} Let $\eta$ be small and {$x\in D(\eta,z,z')$} where
$$
{D(\eta,z,z') = \Omega(\eta,z) \cap \Omega(\eta,z').}
$$
\begin{itemize}
\item For all $z,z'\in \Sigma$ {and sufficiently small $\eta$}, there exists a constant $\delta$ independent of $k$ such that
$$
\Im \psi\left(x,z,z'\right) \geq\  \delta\left[\left|x- \gamma \right|^2+\left|x-\gamma' \right|^2\right].
$$
\item For $|\gamma(x; z)-\gamma(x; z')|< \theta |z-z'|$,
\begin{align*}
 |\nabla_x\psi(x,z,z')| \geq C(\theta,\eta)|z-z'|\ ,
\end{align*}
where $C(\theta,\eta)$ is independent of $x$
and positive if $\theta$ and $\eta$ are
sufficiently small.
\end{itemize}
\end{lem}
\begin{proof} The first result 
{follows directly from \eq{imphiest}.}
For the second result, 
we proceed to obtain
\begin{align*}
 |\nabla_x\psi(x,z,z')|
  \geq{}& |\Re\nabla_x\psi(x,z,z')| \\
  ={}& |\Re\nabla_x\phi(x; z')-\Re\nabla_x\phi(x; z)|, \qquad \Bigl\{h:= \Re\nabla_x\phi\Bigr\} \\
  ={}& \Bigl|h(\gamma'; z')-h(\gamma; z)+ h(\gamma; z') -h(\gamma';z')  \\
  & + h(x;z')- h( \gamma; z') +h( \gamma, z)-h(x, z)\Big|.
\end{align*}
{For the function $z\mapsto s(x;z)$ we can find a Lipschitz constant
that is uniform in $x$. 
Recalling that $\gamma=x(s(x;z);z)$
and $\gamma'=x(s(x;z');z')$
we can therefore use 
(\ref{NonSqueezeIneq}) in Lemma \ref{nonsqueezing}
for the first pair, and obtain}
\begin{align*}
 | h(\gamma'; z')-h(\gamma; z)|=|p(s(x;z); z') -p(s(x;z');z)|
  \geq c_1|z-z'| - |\gamma-\gamma'|.
\end{align*}
The second pair $|h(\gamma; z') -h(\gamma';z')|$ 
is  bounded by $C_1 |\gamma-\gamma'|$.
Then, by the Fundamental Theorem of Calculus, for
{$x\in D(\eta,z,z')$}, the remaining terms are
\begin{align*}
\left| \int_0^1 \left[ D^2\phi(\tau x+(1-\tau) \gamma;  z') - D^2\phi(\tau x+(1-\tau) \gamma; z) \right] (x- \gamma) d\tau  \right| 
\leq C|z-z'||x- \gamma| \leq C_2\eta |z-z'|.
\end{align*}
 Using these estimates for the case $|\gamma-\gamma'|<  \theta |z-z'|$ we then obtain
\begin{align*}
 |\nabla_x\psi(x,z,z')|
  &\geq c_1|z-z'| - |\gamma-\gamma'| - C_1|\gamma-\gamma'|  - C_2\eta|z-z'| \\
  &\geq c_1|z-z'| - (1+C_1)\theta|z-z'| - C_2\eta |z-z'| \\
  &=: C(\theta,\eta)|z-z'|\ ,
\end{align*}
where $C(\theta,\eta)$ is positive if $\theta$
and $\eta$ are small enough.
\end{proof}

\subsection{Estimate of $I_1$}

We start by looking at $I_1$ which
 corresponds to the non-caustic region of the solution.
We have
\begin{align*}
{\mathcal I}_1 &:=\left| \int_{|x|\leq 5R}\int_{\Sigma}\int_{\Sigma} I_1(x,z, z') dzdz'dx  \right|\\
 &\leq  k^{1+|\beta|}
 {\int_{\Sigma}\int_{\Sigma}\int_{D(\eta,z,z')}}
  \chi_1(x,z,z')
   e^{ik \psi(x,z,z')} g(x; z') \overline{g(x;z)} \\
 & \qquad\qquad \times (x-\gamma)^\beta (x-\gamma')^\beta  \varrho_\eta(x-\gamma) \varrho_\eta(x-\gamma')dxdzdz'.
\end{align*}
We begin estimating
\begin{align*}
 \left|{\mathcal I}_1\right|
 &\leq
 C k^{1+|\beta|}
  {\int_{\Sigma}\int_{\Sigma}\int_{D(\eta,z,z')}}
  \chi_1(x,z,z')
  |x-\gamma|^{|\beta|} |x-\gamma'|^{|\beta|}e^{-\delta k (|x-\gamma||^2+|x-\gamma|'^2)}dxdzdz'.
\end{align*}
 Now, using the estimate \eq{expest}
with $p=|\beta|$, $a=\delta
k$ and $s=|x-\gamma|$ or $|x-\gamma'|$, and
continuing the estimate of $I_1$, we have for a
constant, $C$, independent of $z$ and $z'$,
\begin{align*}
 \left|{\mathcal I}_1\right| &\leq C k^{|\beta|+1} \left(\frac{1}{k\delta}\right)^{|\beta|}
 {\int_{\Sigma}\int_{\Sigma}\int_{D(\eta,z,z')}}
   \chi_1(x,z,z')
 e^{-\frac{\delta k}{2}(|x-\gamma|^2+|x-\gamma'|^2)} \ dxdzdz' \\
   &\leq C k
      {\int_{\Sigma}\int_{\Sigma}\int_{D(\eta,z,z')}}
      \chi_1(x,z,z')
      e^{-\frac{\delta k}{4}(|x-\gamma|^2+|x-\gamma'|^2)} e^{-\frac{\delta k}{8}|\gamma-\gamma'|^2} e^{-\frac{\delta k}{2}  |x-\bar \gamma|^2|} \ dxdzdz'\\
   &\leq C k
      {\int_{\Sigma}\int_{\Sigma}\int_{D(\eta,z,z')}}
      \chi_1(x,z,z')
   e^{-\frac{\delta k}{4}(|x-\gamma|^2+|x-\gamma'|^2)} e^{-\frac{\delta k}{8} |\gamma-\gamma'|^2} \
   dxdzdz'  \\
   &\leq C k
      {\int_{\Sigma}\int_{\Sigma}e^{-\frac{\delta k}{8} \theta|z-z'|^2} \int_{D(\eta,z,z')}}
   e^{-\frac{\delta k}{4}(|x-\gamma|^2+|x-\gamma'|^2)}  \ dx dzdz'.
\end{align*}
Here we have used the identity
$$
|x-\gamma|^2+|x-\gamma'|^2=2|x-\bar \gamma|^2+\frac{1}{2}|\gamma-\gamma'|^2,
$$
and the fact that
$|\gamma-\gamma'|>\frac12\theta|z-z'|$ on the
support of $\chi_1$. For the inner integral we
can use Cauchy--Schwarz, {together with the fact that $D\subset\Omega(\eta;z)$
and $D\subset\Omega(\eta;z')$,}
\begin{align*}
{\int_{D(\eta,z,z')}}
   e^{-\frac{\delta k}{4}(|x-\gamma|^2+|x-\gamma'|^2)}dx &\leq
   \left(
   \int_{{\Omega(\eta;z)}}
   e^{-\frac{\delta k}{2}(|x-\gamma|^2)}dx
   \int_{{\Omega(\eta;z')}}
   e^{-\frac{\delta k}{2}(|x-\gamma'|^2)}dx\right)^{1/2}.
\end{align*}
By a change of local coordinates we can show that
\begin{equation}\label{degat}
\int_{{\Omega(\eta;z)}} e^{-\frac{\delta k}{4}|x-\gamma|^2}  \ dx \leq Ck^{(1-d)/2}.
\end{equation}
From this it follows that
\begin{equation}\label{I1x}
|{\mathcal I}_1|\leq C k^{(3-d)/2}
      \int_{\Sigma}\int_{\Sigma}
   e^{-\frac{\delta k}{8} \theta|z-z'|^2} \ dzdz'.
\end{equation}
To show (\ref{degat})  for each $z$,  we
introduce local coordinates in  the tubular
neighborhood {$\Omega(\eta;z)$ around the ray $\gamma$} in the
following way:  choose (smoothly in $(s,z)$) a
normalized  orthogonal basis $e_1(s,z),\ldots,
e_{d-1}(s,z)$ in the plane $\{x\,:\,
(x-{x(s;z)})\cdot p(s; z)=0\}$ with the origin
at {$x(s;z)$}. Since $s$ and $z$ lie in
compact sets, there will be an $\eta>0$ such that
in the tube {$\Omega(\eta;z)$}
the mapping from $x$ to $(s, y)$ defined by
$$
x={x(s;z)}+y_1 e_1(s,z)+\dots+ y_{d-1}\cdot e_{d-1}(s,z)
$$
will be a diffeomorphism depending smoothly on
$z$, hence
\begin{align*}
\int_{{\Omega(\eta;z)}} e^{-\frac{\delta k}{4}|x-\gamma|^2}  \ dx =
{\int_{|s|\leq L_0}\int_{|y|\leq \eta}} e^{-\frac{\delta k}{4}|y|^2}\left| \frac{\partial x}{\partial(y, s)}\right|dy ds
\leq Ck^{(1-d)/2},
\end{align*}
{where $L_0$ is chosen such that $|x(L_0;z)|\geq 5R$ for all $z\in\Sigma$.}
Letting $\Lambda =\sup_{z,z'\in \Sigma}|z-z'|<
\infty$ be the diameter of $\Sigma$, we continue
to estimate the $(z,z')$-integral  left  in
(\ref{I1x}):
\begin{align*}
|{\mathcal I}_1| &\leq C k^{(3-d)/2}
      \int_{\Sigma}\int_{\Sigma}
   e^{-\frac{\delta k}{8} \theta|z-z'|^2} \ dzdz'\\
 &\leq  C k^{(3-d)/2}\int_0^\Lambda \tau^{d-2}e^{-\frac{k \delta\theta^2}{8}\tau^2} d\tau\\
&\leq C k^{2-d},
\end{align*}
which concludes the estimate of $I_1$.

\subsection{Estimate of $I_2$}

In order to estimate $I_2$ we use a version of
the non-stationary phase lemma (see
\cite{Hormander}).

\begin{lem}[Non-stationary phase lemma]\label{statphase}
Suppose that $u(x;\zeta)\in C_0^\infty(\Omega
\times Z)$, where $\Omega$ and $Z$ are compact
sets and ${\psi(x; \zeta)\in C^\infty(O)}$ for
some open neighborhood $O$ of $\Omega \times Z$.
If $\nabla_x \psi$ never vanishes in $O$, then
for any $K=0,1,\ldots$,
\begin{align*}
   \left|\int_\Omega u(x; \zeta)e^{i k\psi(x; \zeta)}dx \right|
   \leq C_K k^{-K}  \sum_{|\lambda|\leq K}\int_\Omega \frac{|\partial_x^{\lambda}u(x; \zeta)|}{|\nabla_x\psi(x;\zeta)|^{2K-|\lambda|}}
   e^{- k \Im \psi(x; \zeta)}dx\ ,
\end{align*}
where $C_K$ is a constant independent of $\zeta$.
\end{lem}
We now define
\begin{align*}
\tilde I_2(z,z') &:= \int_{|x|\leq 5R} I_2(x,z,z')dx\\
&= k^{1+|\beta|}
  \int_{{D(\eta,z,z')}}\chi_2(x,z,z') e^{ik \psi(x,z,z')} g(x; z') \overline{g(x;z)} \\
 & \qquad\qquad\qquad\qquad \times (x-\gamma)^\beta (x-\gamma')^\beta  \varrho_\eta(x-\gamma) \varrho_\eta(x-\gamma')dx.
\end{align*}
In this case, non-stationary phase Lemma
\ref{statphase} can be applied to $\tilde I_2$
with $\zeta=(z,z')\in  \Sigma \times \Sigma $ to
give,
\begin{align*}
 \left|\tilde I_2\right|
 &\leq C_K k^{1+|\beta|-K}  \sum_{|\lambda| \leq K} \int_{{D(\eta,z,z')}}\frac{\left|\partial^\lambda_x \left[(x-\gamma)^\beta(x-\gamma')^\beta \chi_2 g' \overline{g}\varrho_\eta\varrho'_\eta\right]\right|}{|\nabla_x\psi(x,z,z')|^{2K-|\lambda|}}e^{-\Im k \psi(x,z,z')} dx \\
 &\leq C_K k^{1+|\beta|-K} \sum_{|\lambda|\leq K}\Bigg(\frac{1}{
 (C(\theta,\eta)|z-z'|)^{2K-|\lambda|}}
\int_{{D(\eta,z,z')}} \left| \partial^\lambda_x\left[(x-\gamma)^\beta(x-\gamma')^\beta \chi_2 g'\overline{g}\varrho_\eta\varrho'_\eta\right]\right
|e^{-\Im k \psi } dx \Bigg) \\
 &\leq C_Kk^{1+|\beta|-K} \sum_{|\lambda|\leq K}\frac{1}{|z-z'|^{2K-|\lambda|}}
\Bigg(\sum_{\substack{\lambda_1+\lambda_2=\lambda\\ \lambda_1\leq2\beta}}\int_{{D(\eta,z,z')}} \left| \partial^{\lambda_1}_y\left[(x-\gamma)^\beta(x-\gamma')^\beta\right]\right|\\
&\qquad\qquad\qquad\qquad\qquad\qquad\qquad\qquad\qquad\qquad\times\left|\partial^{\lambda_2}_x\left[ \chi_2 g'\overline{g}\varrho_\eta\varrho'_\eta\right]\right|e^{-\Im k\psi} dx \Bigg) \ ,
\end{align*}
where $\varrho'_\eta= \varrho_\eta(x-\gamma')$, and we
used the fact that $|\nabla_x\psi(x,z,z')|\geq
C(\theta,\eta)|z-z'|$ on the support of $\chi_2$, {shown in 
Lemma \ref{PhaseEst}.}
The constant $C_K$ is independent of $z$ and
$z'$. By the bound (\ref{g}) and since
$\varrho_\eta$ is uniformly smooth and $x$, $z$,
$z'$ vary in a compact set,
$\left|\partial^{\lambda_2}_x\left[\chi_2
g'\overline{g}\varrho_\eta\varrho'_\eta\right]\right|$
can be bounded by a constant independent of $x$,
$z$ and $z'$. We estimate the other term as
follows,
\begin{align*}
 \left| \partial_x^{\lambda_1} \left[ (x-\gamma)^\beta (x-\gamma')^\beta \right]\right|
  &\leq C  \sum_{\substack{\lambda_{11}+\lambda_{12}=\lambda_1\\ \lambda_{11},\lambda_{12}\leq\beta}}\left| (x-\gamma)^{\beta-\lambda_{11}}(x-\gamma')^{\beta-\lambda_{12}}\right| \\
 &\leq C \sum_{\substack{\lambda_{11}+\lambda_{12}=\lambda_1\\ \lambda_{11},\lambda_{12}\leq\beta}}|x-\gamma|^{|\beta|-|\lambda_{11}|}\ |x-\gamma'|^{|\beta|-|\lambda_{12}|}\ .
\end{align*}
Now, using the same argument as for estimating
$I_1$, we have
\begin{align*}
& \int_{{D(\eta,z,z')}} \left| \partial^{\lambda_1}_y\left[(x-\gamma)^\beta(x-\gamma')^\beta\right]\right|\left|\partial^{\lambda_2}_y\left[ \chi_2 g'\overline{g}\varrho_\eta\varrho'_\eta\right]\right|e^{-\Im k \psi} dx\\
&\quad\qquad\leq C \sum_{\substack{\lambda_{11}+\lambda_{12}=\lambda_1\\ \lambda_{11},\lambda_{12}\leq\beta}}\int_{{D(\eta,z,z')}} |x-\gamma|^{|\beta|-|\lambda_{11}|}\ |x-\gamma'|^{|\beta|-|\lambda_{12}|}e^{-\Im k \psi } dx \\
&\quad\qquad\leq C(\lambda_2) k^{\frac{ - |\beta|+|\lambda_{11}|-|\beta|+|\lambda_{12}|}{2}}
\int_{{D(\eta,z,z')}} e^{-\frac{k\delta}{2}((x-\gamma)^2+(x-\gamma')^2)}dx\\
&\quad\qquad  \leq  C  k^{(1-d)/2 -|\beta|+|\lambda_1|/2} \ ,
\end{align*}
and consequently,
\begin{align*}
\left|\tilde I_2\right|
& \leq C_K k^{1+|\beta|-K} \sum_{|\lambda|\leq K}\frac{1}{|z-z'|^{2K-|\lambda|}}\sum_{\substack{\lambda_1+\lambda_2=\lambda\\ \lambda_1\leq2\beta}}C(\lambda_2) k^{(1-d)/2 - |\beta|+|\lambda_1|/2} \\
& \leq C_K k^{(3-d)/2}\sum_{|\lambda|\leq K}\frac{1}{(|z-z'|\sqrt{k})^{2K-|\lambda|}} \ .
\end{align*}
On the support of $\chi_2$ the difference
$|z-z'|$ can be arbitrary small, in which case
this estimate is not useful. However, it is easy
to check that the estimate is true also for
$K=0$, and $\tilde{I}_2$ is thus bounded by the
minimum of the $K=0$ and $K>0$ estimates.
Therefore,
\begin{align*}
\left|\tilde I_2\right|
&\leq C k^{(3-d)/2}\min\left[1,\sum_{|\lambda|\leq K}\frac{1}{\left(|z-z'|\sqrt{k}\right)^{2K-|\lambda|}}\right] \\
 &\leq C k^{(3-d)/2}\sum_{|\lambda|\leq K}\min\left[1,\frac{1}{\left(|z-z'|\sqrt{k}\right)^{2K-|\lambda|}}\right] \\
 &\leq C k^{(3-d)/2}\sum_{|\lambda|\leq K}\frac{1}{1+\left(|z-z'|\sqrt{k}\right)^{2K-|\lambda|}} \leq C \frac{k^{(3-d)/2}}{1+\left(|z-z'|\sqrt{k}\right)^{K}} \ .
\end{align*}
Finally, letting $\Lambda =\sup_{z,z'\in
\Sigma}|z-z'|< \infty$ be the diameter of
$\Sigma$, we compute
\begin{align*}
\int_{\Sigma\times \Sigma }\left|\tilde I_{2}(z,z')\right| dzdz'
&\leq  C k^{\frac{3-d}{2}}\int_{\Sigma \times \Sigma} \frac{1}{1 + \left(|z-z'|\sqrt{k}\right)^K} dzdz'\\
&\leq  C k^{\frac{3-d}{2}}\int_0^\Lambda \frac{1}{1 + (\tau \sqrt{k})^K}\tau^{d-2} d\tau \\
& \leq C k^{2-d} \int_0^\infty \frac{\xi^{d-2}}{1+\xi^K}d\xi \\
&\leq  C k^{2-d} \ ,
\end{align*}
if we take $K>d-1$. This shows the $I_2$
estimate,  which proves claim (\ref{Izz}).

\section{Another Superposition}

Specializing to $\rho(x)=(x-y)\cdot \nu$, one can
also take the superposition with respect to
$\nu$. We will carry this out for $d=3$. Starting
with an inversion formula for the Radon
transform:

$$f(x)=-{1\over 8\pi^2}\Delta \left(\int_{S^2}d\nu\left(\int_{(x-y)\cdot \nu=0}f(y)dA_y\right)\right),$$
and noting that $\int_{S^2}d\nu\int_{(x-y)\cdot
\nu=0}f(y)dA_y$ tends to zero as $|x|\to\infty$
when $f\in C_c(\Bbb R^3)$, it follows that
$$\int_{S^2}d\nu\left(\int_{(x-y)\cdot \nu=0}f(y)dA_y\right)=2\pi \int_{\Bbb R^3}{f(y)\over |x-y|}dy.$$
In other words
$$\int_{S^2}\delta(x\cdot\nu)d\nu={2\pi\over
|x|}$$ as a distribution. Hence, ignoring $\rho$ and the lower order term
$$\int_{S^2}g(\nu,y,k)d\nu={2\pi i\over
k}{e^{-k|x-y|^2}\over |x-y|}=_{def}h(x;y,k),$$ and
$\int_{S^2}u_{GB}(x;\nu,y)d\nu$ is a
approximation to the outgoing solution to
$L_nu=h$ satisfying the estimate \eqref{E}.

\section*{Acknowledgments}
This article arose from work at the SQuaRE
project ``Gaussian beam superposition methods for
high frequency wave propagation" supported by the
American Institute of Mathematics (AIM), the
authors acknowledge the support of AIM and the
NSF.

\appendix

\section{Form of the Green's Function}

Let $G_\lambda(x)$ be the free space Green's function for the Helmholtz equation
at {complex valued}
wave number $\lambda=|\lambda|\beta$ where $\beta$ is complex number with
$|\beta|=1$ and $\Im \beta\geq 0$. The Green's function has
the following properties,
\begin{equation}\label{twop}
G_\lambda(x)=O(e^{-\Im k|x|} |x|^{\frac{1-d}{2}}), \quad \partial_r G_\lambda(x) -i\lambda G_\lambda(x)= {O(|x|^{\frac{1-d}{2}})}, \quad r=|x|\to \infty.
\end{equation}
The dependence on {$|k|$} can be scaled out and by rotational invariance
we can write $G_\lambda(x) = |\lambda|^{d-2}\bar{G}_\beta(|\lambda x|)$ where $G_\beta(x)=\bar{G}_\beta(|x|)$.
Then, if
$$
   \bar{G}_\beta(r) = \frac{e^{i\beta r}}{(\beta r)^{\frac{d-1}{2}}} \bar{w}_\beta(r),
$$
the complex valued 
function $\bar{w}_\beta$ will satsify the following ODE for $r>0$,
\begin{equation}\label{ode}
\bar{w}_\beta''(r)+2i\beta\bar{w}_\beta'(r)
- \frac{c_d}{r^{2}}\bar{w}_\beta(r)=0,\qquad
c_d= \left(\frac{d-2}{2}\right)^2-\frac14,
\end{equation}
This follows from applying the Helmholtz operator in $d$ dimensions
to ${G}_\beta$ away from $x=0$ (with $r=|x|$),
\begin{align*}
0&=
\Delta G_\beta(x)+ \beta^2G_\beta(x)
=
 \frac{d^2}{dr^2}\bar{G}_\beta(r)
  + \frac{d-1}{r}\frac{d}{dr}\bar{G}_\beta(r)+\beta^2\bar{G}_\beta(r)\\
  &=\frac{e^{i\beta r}}{(\beta r)^{\frac{d-1}{2}}}\left(
\bar{w}_\beta''(r)+2i\beta\bar{w}_\beta'(r)
- \frac{(d-1)(d-3)}{4}\frac{\bar{w}_\beta(r)}{r^{2}}
\right).
\end{align*}
After differentiating (\ref{ode}) $p$ times we get
\begin{equation}\label{dode}
   \bar{w}_\beta^{(p+2)}(r) + 2i\beta\bar{w}_\beta^{(p+1)}(r) +\sum_{j=0}^p d_{p,j} \bar{w}_\beta^{(j)}(r) r^{-2-p+j}=0,
\end{equation}
for some coefficients $d_{p,j}$. 
From the left property in (\ref{twop}) it follows that {$|\bar{w}_\beta(r)|\leq B_0$
for some bound $B_0$ and $r>1$.}
Moreover, the right property 
(the radiation condition) implies that 
$\bar{w}_\beta'\to (d-1)\bar{w}_\beta/2r$ as $r\to\infty$.
It then follows by induction on (\ref{dode}) that
$\bar{w}_\beta^{(p)}(r)\to 0$ for all $p\geq 1$.

We now claim that there are bounds $B_p$, independent of $r$, such that
$|r^p\bar{w}_\beta^{(p)}(r)|\leq B_p$ for $r>1$. We just saw that
this is true for $p=0$ and we make the induction hypothesis 
that it is true for $j=0,\ldots, p$. Then from (\ref{dode}), 
\begin{align*}
  \left| \frac{d}{dr} e^{2i\beta r}\bar{w}_\beta^{(p+1)}(r)\right|
  &=e^{-2r\Im\beta}\left|\bar{w}_\beta^{(p+2)}(r)+2i\beta\bar{w}_\beta^{(p+1)(r)}\right|\leq 
  e^{-2r\Im\beta}\sum_{j=0}^p|d_{p,j}||\bar{w}_\beta^{(j)}(r)| r^{-2-p+j}\\  
  &\leq B'_{p+1} e^{-2r\Im\beta} r^{-2-p},
\end{align*}
when $r>1$,
where $B'_{p+1} = \sum_{j=0}^p|d_{p,j}B_j|$. Since $\bar{w}_\beta^{(p+1)}(r)\to 0$ 
as $r\to\infty$ and $\Im \beta\geq 0$,
\begin{align*}
  \left|\bar{w}_\beta^{(p+1)}(r)\right| &= e^{2r\Im\beta}\left|\int_r^\infty \frac{d}{ds} e^{2i\beta s}\bar{w}^{(p+1)}(s) ds\right|
  \leq B'_{p+1}\int_r^\infty \frac{e^{2(r-s)\Im\beta}}{s^{p+2}}ds
  \leq \int_r^\infty \frac{B'_{p+1}}{s^{p+2}}ds
  = \frac{B_{p+1}}{r^{p+1}},
\end{align*}
where $B_{p+1}  = B_{p+1}'/(p+1)$.
This shows the claim. 

We conclude that
$$
  G_\lambda(x) = |\lambda|^{d-2}\bar{G}_\beta(|\lambda x|) = 
  \frac{e^{i\lambda|x|}}{|x|^{\frac{d-1}2}}w(x;\lambda),\qquad w(x;\lambda) = |\lambda|^{\frac{d-3}2}\beta^{\frac{1-d}2}\bar{w}_\beta(|\lambda x|),
$$
and for any multi-index $\alpha$,
\begin{align*}
  |\partial_x^\alpha w(x;{\lambda})| &\leq
   C|{\lambda}|^{\frac{d-3}2} \sum_{j=0}^{|\alpha|}
  \left|\frac{d^j}{dr^j} \bar{w}_\beta(|{\lambda}|r)\right|_{r=|x|}
=   |{\lambda}|^{\frac{d-3}2} \sum_{j=0}^{|\alpha|}
  \left|{\lambda}^j\bar{w}_\beta^{(j)}({\lambda}|x|)\right|
=   |{\lambda}|^{\frac{d-3}2} \sum_{j=0}^{|\alpha|}
  B_j|x|^{-j}\\
  &\leq C(\delta)|{\lambda}|^{\frac{d-3}2},
\end{align*}
when $|x|>\delta$ and $|{\lambda}|>1/\delta$.


\begin{thebibliography}{1}


\bibitem{AETT:2010}
G.~Ariel, B.~Engquist, N.~M. Tanushev, and
R.~Tsai.
\newblock Gaussian beam decomposition of high frequency wave fields using
  expectation-maximization.
\newblock {\em J. Comput. Phys.}, 230(6):2303--2321, 2011.

\bibitem{BabicBuldyrev:1991}
V.~M. Babi\v{c} and V.~S. Buldyrev.
\newblock {\em Short-Wavelength Diffraction Theory: Asymptotic Methods},
  volume~4 of {\em Springer Series on Wave Phenomena}.
\newblock Springer-Verlag, 1991.

\bibitem{Babic:1973}
V.~M. Babi\v{c} and T.~F. Pankratova.
\newblock On discontinuities of {G}reen's function of the wave equation with
  variable coefficient.
\newblock {\em Problemy Matem. Fiziki}, 6, 1973.
\newblock Leningrad University, Saint-Petersburg.

\bibitem{BabicPopov:1989}
V.~M. Babi\v{c} and M.~M. Popov.
\newblock Gaussian summation method (review).
\newblock {\em Izv. Vyssh. Uchebn. Zaved. Radiofiz.}, 32(12):1447--1466, 1989.

\bibitem{BenColRun:04}
J.-D. Benamou, F.~Collino, and O.~Runborg.
\newblock Numerical microlocal analysis of harmonic wavefields.
\newblock {\em J. Comput. Phys.}, 199(2):717--741, 2004.

\bibitem{Bl84} 
N.~ Bleistein.
\newblock Mathematical methods for wave phenomena.
\newblock Academic Press, INC. 1984. 

\bibitem{BougachaEtal:09}
S.~Bougacha, J.-L. Akian, and R.~Alexandre.
\newblock {G}aussian beams summation for the wave equation in a convex domain.
\newblock {\em Commun. Math. Sci.}, 7(4):973--1008, 2009.

\bibitem{CastellaJecko:06}
F. Castella and T. Jecko. Besov estimates in the
high-frequency Helmholtz equation, for a
non-trapping and $C^2$ potential. {\em J. Diff.
Eq.}, 228(2):440--485, 2006.

\bibitem{CasPerRun:02}
F.~Castella, B.~Perthame, and O.~Runborg.
\newblock High frequency limit of the {H}elmholtz equation {I}{I}: Source on a
  general smooth manifold.
\newblock {\em Commun.\ Part.\ Diff.\ Eq.}, 27:607--651, 2002.

\bibitem{Cerveny_etal:1982}
V.~C. \v{C}erven\'{y}, M.~M. Popov, and
I.~P\v{s}en\v{c}\'{i}k.
\newblock Computation of wave fields in inhomogeneous media --- {G}aussian beam
  approach.
\newblock {\em Geophys. J. R. Astr. Soc.}, 70:109--128, 1982.

\bibitem{EngRun:03}
B.~Engquist and O.~Runborg.
\newblock Computational high frequency wave propagation.
\newblock {\em Acta Numerica}, 12:181--266, 2003.


\bibitem{Erlaga1}
Y.~A. Erlangga.
\newblock Advances in iterative methods and preconditioners for the {H}elmholtz
  equation.
\newblock \emph{Arch. Comput. Methods Eng.}, 15:\penalty0 37--66, 2008.

\bibitem{FaouLubich:06}
E.~Faou and C.~Lubich.
\newblock A Poisson integrator for gaussian wavepacket dynamics.
\newblock {\em Computing and Visualization in Science}, 9(2):45--55, 2006.

\bibitem{Hagedorn:80}
G.~A. Hagedorn.
\newblock Semiclassical quantum mechanics. {I}. {T}he {$\hbar\rightarrow 0$}
  limit for coherent states.
\newblock {\em Comm. Math. Phys.}, 71(1):77--93, 1980.

\bibitem{Heller:75}
E.~J. Heller.
\newblock Time-dependent approach to semiclassical dynamics.
\newblock {\em J. Chem. Phys.}, 62(4):1544--1555, 1975.

\bibitem{Heller:81}
E.~J. Heller.
\newblock Frozen {G}aussians: a very simple semiclassical approximation.
\newblock {\em J. Chem. Phys.}, 76(6):2923--2931, 1981.

\bibitem{HermanKluk:84}
M.~F. Herman and E.~Kluk.
\newblock A semiclassical justification for the use of non-spreading
  wavepackets in dynamics calculations.
\newblock {\em Chem. Phys.}, 91(1):27--34, 1984.

\bibitem{Hill:1990}
N.~R. Hill.
\newblock Gaussian beam migration.
\newblock {\em Geophysics}, 55(11):1416--1428, 1990.

\bibitem{Hill:2001}
N.~R. Hill.
\newblock Prestack {G}aussian beam depth migration.
\newblock {\em Geophysics}, 66(4):1240--1250, 2001.

\bibitem{HormanderFIO:1971}
L.~H{\"o}rmander.
\newblock Fourier integral operators. {I}.
\newblock {\em Acta Math.}, 127(1-2):79--183, 1971.

\bibitem{Hormander:71}
L.~H{\"o}rmander.
\newblock On the existence and the regularity of solutions of linear
  pseudo-differential equations.
\newblock {\em L'Enseignement Math\'ematique}, XVII:99--163, 1971.

\bibitem{Hormander}
L.~H{\"o}rmander, {\em The Analysis of Linear
Partial Differential Operators I: Distribution
Theory and Fourier Analysis}, Springer-Verlag,
Berlin Heidelberg New York, 1983.



\bibitem{JinMarSpa:12}
S.~Jin, P.~Markowich, and C.~Sparber.
\newblock Mathematical and computational models for semiclassical
  {S}chr\"odinger equations.
\newblock {\em Acta Numerica}, pages 1--89, 2012.

\bibitem{JinWuYang:08}
S.~Jin, H.~Wu, and X.~Yang.
\newblock Gaussian beam methods for the {S}chr\"odinger equation in the
  semi-classical regime: {L}agrangian and {E}ulerian formulations.
\newblock {\em Commun. Math. Sci.}, 6:995--1020, 2008.

\bibitem{JinEtal:10}
S.~Jin, H.~Wu, X.~Yang, and Z.~Y. Huang.
\newblock Bloch decomposition-based {G}aussian beam method for the
  {S}chr\"odinger equation with periodic potentials.
\newblock {\em J. Comput. Phys.}, 229(13):4869--4883, 2010.


\bibitem{Jin2}
S.~Jin, H.~Wu and X.~Yang, {\em A Numerical Study
of the Gaussian Beam Methods for One-Dimensional
Schr\"odinger-Poisson Equations}, J. Comp. Math.,
to appear.

\bibitem{Katchalov_Popov:1981}
A.~P. Katchalov and M.~M. Popov.
\newblock Application of the method of summation of {G}aussian beams for
  calculation of high-frequency wave fields.
\newblock {\em Sov. Phys. Dokl.}, 26:604--606, 1981.

\bibitem{KKL01}
A.~Katchalov, Y.~Kurylev and M. ~Lassas.
\newblock Inverse boundary spectral problems, Chapman and Hall (2001)

\bibitem{Keller:62}
J.~Keller.
\newblock Geometrical theory of diffraction.
\newblock \emph{J. Opt. Soc. Amer}, 52, 1962.


\bibitem{Klimes:1984}
L.~Klime\v{s}.
\newblock Expansion of a high-frequency time-harmonic wavefield given on an
  initial surface into {G}aussian beams.
\newblock {\em Geophys. J. R. astr. Soc.}, 79:105--118, 1984.


\bibitem{LeungQian:09}
S.~Leung and J.~Qian.
\newblock Eulerian {G}aussian beams for {S}chr\"odinger equations in the
  semi-classical regime.
\newblock {\em J. Comput. Phys.}, 228:2951--2977, 2009.

\bibitem{LeuQiaBur:07}
S.~Leung, J.~Qian, and R.~Burridge.
\newblock Eulerian {G}aussian beams for high frequency wave propagation.
\newblock {\em Geophysics}, 72:SM61--SM76, 2007.

\bibitem{LuYang:12}
J.~Lu and X.~Yang.
\newblock Convergence of frozen Gaussian approximation for high frequency wave propa-
gation.
\newblock {\em Comm. Pure Appl. Math.}, 65:759--789, 2012.


\bibitem{LiuRalston:09}
H.~Liu and J.~Ralston.
\newblock Recovery of high frequency wave fields for the acoustic wave
  equation.
\newblock {\em Multiscale Model. Sim.}, 8(2):428--444, 2009.

\bibitem{LiuRalston:10}
H.~Liu and J.~Ralston.
\newblock Recovery of high frequency wave fields from phase space--based
  measurements.
\newblock {\em Multiscale Model. Sim.}, 8(2):622--644, 2010.


\bibitem{LRT11}
H.~Liu, O.~Runborg, and N.~M. Tanushev.
\newblock Error estimates for {G}aussian beam superpositions.
\newblock {\em Math. Comp.}, 82:919--952, 2013.


\bibitem{MR78}
A.~Majda, and J.~Ralston.
\newblock An analogue of Weyl's theorem for
unbounded domains, II.
\newblock {\em Duke Math. Journal}, 45:183--196, 1978.



\bibitem{MotamedRunborg:09}
M.~Motamed and O.~Runborg.
\newblock Taylor expansion and discretization errors in {G}aussian beam
  superposition.
\newblock {\em Wave Motion}, 2010.




\bibitem{Ra82}
J.~Ralston.
\newblock Gaussian beams and the propagation of singularities.
\newblock In {\em Studies in partial differential equations}, volume~23 of {\em
  MAA Stud. Math.}, pages 206--248. Math. Assoc. America, Washington, DC, 1982.

\bibitem{PerthameVega:99}
B.~Perthame and L.~Vega.
\newblock Morrey--{C}ampanato estimates for {H}elmholtz equations.
\newblock {\em Journal of Functional Analysis}, 164:340--355, 1999.

\bibitem{Popov:1982}
M.~M. Popov.
\newblock A new method of computation of wave fields using {G}aussian beams.
\newblock {\em Wave Motion}, 4:85--97, 1982.

\bibitem{QianYing:2010}
J.~Qian and L.~Ying.
\newblock Fast {G}aussian wavepacket transforms and {G}aussian beams for the
  {S}chr\"odinger equation.
\newblock {\em J. Comput. Phys.}, 229:7848--7873, 2010.


\bibitem{RousseSwart:09}
V.~Rousse and T.~Swart.
\newblock A mathematical justification for the {H}erman--{K}luk propagator.
\newblock {\em Comm. Math. Phys.}, 286(2):725--750, 2009.

\bibitem{Runborg:07}
O.~Runborg.
\newblock Mathematical models and numerical methods for high frequency waves.
\newblock {\em Commun. Comput. Phys.}, 2:827--880, 2007.

\bibitem{Tanushev:08}
N.~M. Tanushev.
\newblock Superpositions and higher order {G}aussian beams.
\newblock {\em Commun. Math. Sci.}, 6(2):449--475, 2008.

\bibitem{TQR:2007}
N.~M. Tanushev, J.~Qian, and J.~V. Ralston.
\newblock Mountain waves and {G}aussian beams.
\newblock {\em Multiscale Model. Simul.}, 6(2):688--709, 2007.

\bibitem{TET:2009}
N.~M. Tanushev, B.~Engquist, and R.~Tsai.
\newblock Gaussian beam decomposition of high frequency wave fields.
\newblock {\em J. Comput. Phys.}, 228(23):8856--8871, 2009.


\bibitem{V75}
B.~Vainberg.
\newblock On short-wave asymptotic behaviour of solutions to steady-state
  problems and the asymptotic behaviour as $t\to\infty$ of solutions of
  time-dependent problems.
\newblock {\em Uspekhi (Russian Math. Surveys)}, 30(2):1--58, 1975.


\bibitem{V89}
B.~R. Vainberg.
\newblock Asymptotic Methods in Equations of
Mathematics Physics, Gordon and Breach (1989)








\end{thebibliography}
\end{document}